\documentclass[reqno,11pt]{amsart}

\numberwithin{equation}{section}
\usepackage{amsmath}
\usepackage{esint} 
\usepackage{cancel}
\usepackage{amsthm}
\usepackage{epsfig}
\usepackage{psfrag}
\usepackage{graphicx}
\graphicspath{{figure/}}
\usepackage{graphpap,latexsym,epsf}
\usepackage{color}
\usepackage{amssymb,mathrsfs,enumerate}
\usepackage{mathtools}
\usepackage{subfigure}
\definecolor{citegreen}{rgb}{0,0.6,0}
\definecolor{refred}{rgb}{0.8,0,0}
\usepackage[colorlinks, citecolor=citegreen, linkcolor=refred]{hyperref}
\newcommand{\R}{\mathbb{R}}
\newcommand{\N}{\mathbb{N}}
\newcommand{\Sph}{\mathbb{S}}

\newcommand{\pa}{\partial}
\newcommand{\Om}{\Omega}

\newcommand{\ep}{\varepsilon}
\newcommand{\rmd}{{\rm d}}

\newcommand{\umax}{u_{{\rm max}}}

\newcommand{\D}{\nabla}
\newcommand{\DD}{\nabla^2}
\newcommand{\De}{\Delta}

\newcommand{\na}{\nabla}
\newcommand{\nana}{\nabla^2}

\mathchardef\emptyset="001F

\definecolor{vgreen}{rgb}{0.1,0.5,0.2}
\definecolor{viola}{RGB}{85,26,139}
\renewcommand{\theequation}{\thesection.\arabic{equation}}

\newtheorem{theorem}{Theorem}[section]
\newtheorem{remark}{Remark}

\newtheorem{definition}{Definition}
\newtheorem{proposition}[theorem]{Proposition}

\newtheorem{notation}{Notation}
\newtheorem{normalization}{Normalization}
\newtheorem{lemma}[theorem]{Lemma}

\newtheorem{thmx}{Theorem}

\begin{document}

\title[On the Serrin problem for ring-shaped domains]
{On the Serrin problem for ring-shaped domains}

\author[V.~Agostiniani]{Virginia Agostiniani}
\address{V.~Agostiniani, Universit\`a degli Studi di Trento,
via Sommarive 14, 38123 Povo (TN), Italy}
\email{virginia.agostiniani@unitn.it}

\author[S.~Borghini]{Stefano Borghini}
\address{S.~Borghini, Universit\`a degli Studi di Milano-Bicocca, via Roberto Cozzi 55,
20126 Milano (MI), Italy}
\email{stefano.borghini@unimib.it}

\author[L.~Mazzieri]{Lorenzo Mazzieri}
\address{L.~Mazzieri, Universit\`a degli Studi di Trento,
via Sommarive 14, 38123 Povo (TN), Italy}
\email{lorenzo.mazzieri@unitn.it}

\begin{abstract} 
In this paper, we deal with the long standing open problem of characterising rotationally symmetric solutions to $\Delta u = - 2$, when Dirichlet boundary conditions are imposed on a {\em ring-shaped} planar domain. From a physical perspective, the solution represents the velocity of a homogeneous incompressible fluid, flowing in steady parallel streamlines
through a {\em hollow} cylindrical pipe and obeying a no-slip condition. In contrast with Serrin's classical result, we show that the simplest possible set of overdetermining conditions, namely the prescription of {\em locally constant} Neumann boundary data, is not sufficient to obtain a complete characterisation of the solutions. A further requirement on the number of maximum points arises in our analysis as a necessary and sufficient condition for the rotational symmetry. 
In fluid-dynamical terms, our results imply that if the wall shear stress is constant on some connected component of the pipe's wall, then the velocity of the fluid must attain its maximal value only at finitely many streamlines, unless the hollow pipe itself consists of a couple of concentric cylindrical round tubes. A major difficulty in the analysis of this problem comes from the lack of monotonicity of the model solutions, which makes the moving plane method ineffective. To remedy this issue, we introduce some new arguments in the spirit of comparison geometry, that we believe of independent interest.
\end{abstract}

\maketitle




\vspace{5pt}

%
%
%
%


\section{Introduction and statement of the main results}
\label{sez_1}



In this paper, we study pairs $(\Om,u)$, where 
$\Om\subset\R^2$ is an open bounded domain with smooth boundary and $u$ is the unique solution to the  Dirichlet problem 
\begin{equation}
\label{eq:prob_SD}
\begin{dcases}
\Delta u=-2 & \quad\mbox{in } \Om,\\
\!\quad u=0 & \quad\mbox{on }\partial \Om.
\end{dcases}
\end{equation}
A classical result due to Serrin (\cite{Serrin}, see also~\cite{Weinberger}) states that, if a pair $(\Omega,u)$ solves~\eqref{eq:prob_SD}, 
and the normal derivative of $u$ at $\pa\Omega$ is constant, then necessarily $\Omega$ is a ball and $u$ is rotationally symmetric. In this case, up to translations and rescaling, the solution is given by 
\begin{equation}
\label{eq:serrin_ST}
\Om_0=B(0,1),
\qquad\qquad
u_0(x)\,=\,\frac{1-|x|^2}{2}.
\end{equation}
For future convenience, we observe that the function $u_0$ achieves its maximal value  
$(u_0)_{\rm max}=1/2$ at the origin, and that 
\begin{equation*}
|\D u_0|^2 \,\, \equiv \,\, 1,   \qquad \hbox{on} \quad \pa \Om_0 \, .
\end{equation*} 
In the following, the couple $(\Om_0, u_0)$, will be referred to as {\em Serrin's solution} to problem~\eqref{eq:prob_SD}. As pointed out in Serrin's original paper~\cite{Serrin}, problem~\eqref{eq:prob_SD} has a nice and insightful fluid-dynamical interpretation. In fact, $u$ can be thought as the velocity of a homogeneous incompressible fluid, flowing in steady laminar flow through a cylindrical pipe of cross section $\Om$, obeying the no-slip condition $u=0$ at the pipe's wall. The normal derivative of $u$ at $\pa \Om$, is also relevant from this point of view, since it is related to the so called {\em wall shear stress}, here denoted by $\tau$, through the formula
\begin{equation*}
\tau \, = \, \mu \, |\nabla u| \,, \qquad \hbox{on} \quad \pa\Om \, .
\end{equation*}
For the sake of simplicity, the dynamic viscosity $\mu$ of the fluid will be assumed to be constantly equal to $1$ throughout the paper, so that the value of $|\nabla u|$ at $\pa \Om$ will be identified with the wall shear stress (WSS for short) exerted by the fluid on the pipe's wall. In such a framework, Serrin's result says that {\em``If the WSS assumes the same value at every point of the boundary, then the pipe's cross section must be a $2$-dimensional round ball, so that the pipe itself must have the geometry of a round cylinder"}.

To introduce in more details the problem of interest here and state our theorems, let us draw the reader's attention on a couple of remarkable features of Serrin's result. The first one is that the {\em connectedness} of the boundary is definitely not an assumption of the theorem. It is instead a consequence of the overdetermining condition
\begin{equation}
\label{eq:overneumann}
 `` \,\, |\nabla u| \,\, 
 {\hbox{is constant on $\pa \Om$ ".}} 
\end{equation}
In other words, the above requirement is strong enough to impose an extremely stringent prescription for the topology of $\pa \Om$. The second feature that we would like to underline is that condition~\eqref{eq:overneumann} also provides the solution $u$ with an extra-property. Indeed, it turns out that $u$ is not only rotationally symmetric, but also {\em monotonically decreasing} with respect to the radial variable.

\subsection{A multiplicity result.} Since the primary aim of the present work is classifying rotationally symmetric solutions to~\eqref{eq:prob_SD}, both the connectedness of the boundary and the monotonicity of the solution should be regarded as much stronger conclusions than desired. Indeed, it is elementary to observe that, whenever $\Om$ is bounded, rotationally symmetric solutions to~\eqref{eq:prob_SD} are completely described, up to translations and rescaling, by the following family of {\em ring-shaped model solutions}, with $0<R<1$:
\begin{equation}
\label{eq:ST}
\Om_R\,=\, \{ r_i(R) < |x| < 1\} \, ,
\qquad\qquad
u_R(x)\,=\,
\frac{1-|x|^2}{2}+R^2\log{|x|}\,,
\end{equation}
where the {\em inner radius} $0<r_i(R)<R$ is the smallest positive 
zero of the function 
\begin{equation*}
(0,+\infty)\ni\rho\longmapsto 1-\rho^2+2R^2\log\rho \, .
\end{equation*}
It is evident (see~\figurename~\ref{fig:profiles}) that the functions $u_R$'s are not monotonically decreasing in $|x|$, and that $\pa\Om_R$ is not connected. The parameter $R$, that we have chosen to describe the above family, will be referred to as the {\em core radius} of the ring-shaped model solution $(\Om_R, u_R)$. It is a natural choice, since the maximum points of $u_R$ -- here denoted by ${\rm MAX}(u_R)$ -- are precisely located at the circle of radius $R$. In particular, we have that
\begin{equation}
\label{eq:umax_ST}
(u_R)_{\rm max}\,=\,
\frac{1-R^2}2+R^2\log R\qquad \hbox{and}\qquad
{\rm MAX}(u_R)=\big\{|x|=R\big\},
\end{equation}
so that
$
\Om_R\setminus{\rm MAX}(u_R)\,=\,\Om_{R,i}\sqcup \Om_{R,o}\,,
$
where 
\begin{equation}
\label{eq:Om+-_ST}
\Om_{R,i}\,=\, \{ r_i(R) < |x| <R \}
\qquad \hbox{and}\qquad
\Om_{R,o}\,=\, \{ R < |x| < 1 \} \, .
\end{equation}
To complete the description of this fundamental family of solutions, let us observe that
\begin{equation}
\label{eq:surfacegravity_ST}
|\nabla u_R|\equiv\displaystyle\frac{R^2- r_i^2(R)}{r_i(R)}, \quad \hbox{on $\Gamma_{R,i}$} \qquad \hbox{and} \qquad |\nabla u_R|\equiv1-R^2, \quad \hbox{on $\Gamma_{R,o}$} \,,
\end{equation}
where $\Gamma_{R,i}=\{|x|=r_i(R)\}$ and $\Gamma_{R,o}=\{|x|= 1\}$ respectively denote the {\em inner} and the {\em outer boundary} of $\Om_R$, so that $\pa \Om_R = \Gamma_{R,i} \sqcup \Gamma_{R,o}$. Notice that Serrin's solution can be recovered as the singular limit of the ring-shaped model solutions, as $R \to 0^+$. Indeed, it can be proven that the functions $u_R$'s converge smoothly to $u_0$ on the compact subsets of $B(0,1) \setminus \{0\}$, so that both the sets ${\rm MAX}(u_R)$ and $\Gamma_{R,i}$ are collapsing onto the origin, as $R \to 0^+$ (see~\figurename~\ref{fig:profiles}).
\begin{figure}
\centering
\includegraphics[scale=0.5]{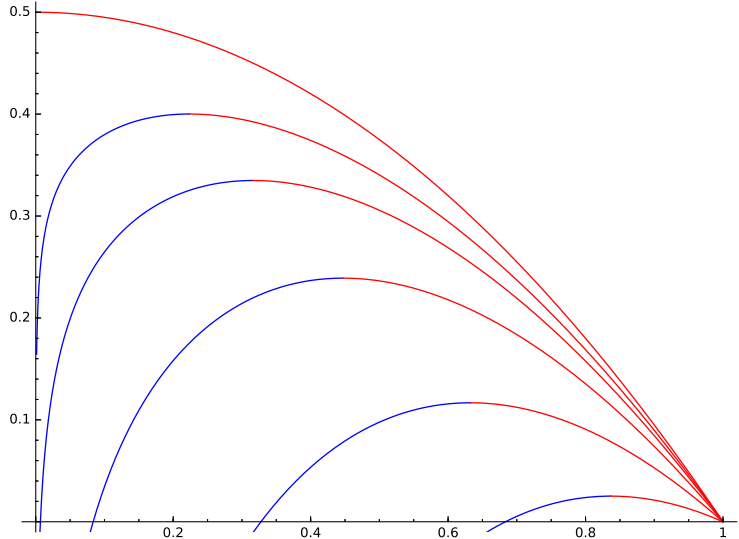}
\caption{The above diagram displays, according to our normalisation, a comparative view of the profiles of the ring-shaped model solutions $u_R$'s, with $0<R<1$, culminating in the profile of Serrin's solution, for $R=0$. The blue and the red parts of the graphs refer to the zones where the profiles are respectively monotonically increasing and decreasing, whereas the parameter $R$ identifies the unique critical point of each profile. It is immediate to notice that for Serrin's solution only the decreasing regime is present, the only critical point being located at the origin.}
\label{fig:profiles}
\end{figure}

Having this picture in mind, it would be tempting to guess that a Serrin-type characterisation of ring-shaped model solutions~\eqref{eq:ST} holds, 
provided 
the overdetermining condition~\eqref{eq:overneumann} is replaced by the requirement that
\begin{equation}
\label{eq:loc_overneumann}
 `` \,\, |\nabla u| \,\,
 {\hbox{is locally constant on $\pa \Om$ ".}} 
\end{equation} 
For the sake of exposition, we focus our attention on the case of {\em ring-shaped domains}, i.e., bounded domains whose boundary has precisely two connected components. To fix the notation, we agree that 
\begin{equation}
\label{eq:gammas}
\pa \Om \, = \, \Gamma_i \sqcup \Gamma_o \, ,
\end{equation}
where $\Gamma_i$ and $\Gamma_o$ denote the inner and the outer connected component of the boundary, respectively.
If $(\Om, u)$ is a solution to problem~\eqref{eq:prob_SD}, we further agree that $\{u=0\}= \pa \Om$ and, for future convenience we set 
\begin{equation*}
 \umax \,\, = \,\, \max_\Om u  \qquad \hbox{and} \qquad {\rm MAX}(u)=\{p\in \Om \, : \, u(p)=\umax\} \, .
\end{equation*}
Our first main result states that, even for ring-shaped domains, condition~\eqref{eq:loc_overneumann} is not strong enough to force the rotational symmetry of $(\Om, u)$.
\begin{thmx}
\label{thm:bifurcation}
There exist infinitely many solutions $(\Om,u)$ to problem~\eqref{eq:prob_SD}, defined on a ring-shaped domain $\Om$, that are not rotationally symmetric and such that $|\D u|$ is locally constant on $\pa\Omega$. 
\end{thmx}
\noindent 
What in fact we are able to prove (see Section~\ref{Sec:bif}) is that there 
exist (infinite) one-parameter families of solutions to~\eqref{eq:prob_SD}, which bifurcate from the family of ring-shaped model solutions~\eqref{eq:ST}.
The proof of the above theorem essentially relies on  the celebrated Crandall-Rabinowitz Bifurcation Theorem~\cite{Cra_Rab}, and it is inspired by the very recent~\cite{Kam_Sci}, where the same result is proven for solutions to problem~\eqref{pb_Syr}, displaying the same (negative) value of the (exterior) normal derivative at all points of the boundary.

\subsection{A partially overdetermined boundary value problem.}

The question arises whether it is possible to equip problem~\eqref{eq:prob_SD} with an overdetermining condition that is powerful enough to select all and only the rotationally symmetric solutions, avoiding on one hand the overkill caused by condition~\eqref{eq:overneumann}, and the multiplicity results allowed on the other hand by condition~\eqref{eq:loc_overneumann}. Answering this question would then provide 
a complete and satisfactory characterisation of the ring-shaped model solutions~\eqref{eq:ST}: a result that, to the best of the authors' knowledge, is so far missing in the literature.
Such a gap might look surprising, if compared with the impressive amount of deep and beautiful works that have been inspired by Serrin's original paper. Dropping any attempt to be complete, we just mention here~\cite{Bar_Mon_Sci,Bra_Nit_Sal_Tro,Bus_Man,Cir_Fig_Ron,
Cir_Vez,Dam_Pac_Ram,Fal_Jar,Gid_Ni_Nir,Ma_Liu,Ser_Zou}. 

As a possible explanation, we observe that the analysis of the case under consideration cannot uniquely rely on the {\em moving plane method},
%
%
%
%
%
since, whenever applicable, it has the drawback of providing the solutions with an undesirable monotonicity. More precisely, a solution fitting into that 
framework would turn out to be nonincreasing in the radial direction, which is false for ring-shaped model solutions~\eqref{eq:ST}, as depicted in~\figurename~\ref{fig:profiles}. 
Consequently, only partial classification results can be obtained in the present context through the moving plane method, for example when the model configuration is represented by the sole outer part $(\Om_{R,o}, {u_R}_{|\Om_{R,o}})$ of a ring-shaped model solution $(\Om_R, u_R)$.
This is the content of the following beautiful theorem, firstly proven by Reichel~\cite[Theorem 3]{Rei_1995} and then refined by Syrakov~\cite[Theorem 2]{Syr_2001}, that also represents so far the most advanced result of the literature along the directions indicated by the present paper.

\begin{theorem}[{\cite[Theorem 2]{Syr_2001}, see also~\cite[Theorem 3]{Rei_1995}}]
\label{thm:Syr}
Let $\Om$ be a ring-shaped domain, let $a>0$ be a positive real number and let $(\Om,u)$ be a solution to the problem 
\begin{equation}
\label{pb_Syr}
\begin{dcases}
\Delta u=-2 & \quad\mbox{in } \Om,\\
\,\,\,\,\, u=0 & \quad\mbox{on } \Gamma_o \, ,\\
\,\,\,\,\, u=a & \quad\mbox{on } \Gamma_{i}\, ,\\
\end{dcases}
\end{equation}
where, according to~\eqref{eq:gammas}, we set $\pa \Om = \Gamma_i \sqcup \Gamma_o$.
Suppose that $|\na u|$ is locally constant
on $\pa\Om$ and that 
\begin{equation}
\label{eq:mp_condition}
\frac{\pa u}{\pa\nu} \, \geq \, 0 \, \qquad \hbox{on
$\Gamma_{i}$} \, , 
\end{equation}
where $\nu$ is the outer unit normal to $\Gamma_i$. Then $(\Om,u)$ is rotationally symmetric and, up to translations and rescaling, it corresponds to a portion of $(\Om_{R,o}, {u_R}_{|\Om_{R,o}})$, for some $0<R<1$. In particular $u$ is nonincreasing in the radial direction.
\end{theorem}

\begin{remark}
Observe that the requirement in~\eqref{eq:mp_condition} is never fulfilled by a solution to~\eqref{eq:prob_SD}, because of the Hopf Lemma.
\end{remark}

In contrast with the above result, we now state our second main theorem, in which a characterisation of the  ring-shaped model solutions~\eqref{eq:ST} on their entire domains is proposed, under the assumption that the set of maximum points of the solution is not finite. A somehow unexpected feature of our theorem is that condition~\eqref{eq:loc_overneumann} is not needed in its full strength, making the problem only partially overdetermined, in the sense of~\cite{Fra_Gaz}.

\begin{thmx}
\label{thm:const_grad}
Let $(\Om,u)$ be a solution to problem~\eqref{eq:prob_SD} such that $\Om$ is a ring-shaped domain and $u$ has infinitely many maximum points. Assume that $|\nabla u|$ is constant on either the inner or the outer boundary component. Then, up to translations and rescaling, $(\Om, u)$ corresponds to a ring-shaped model solution~\eqref{eq:ST}.
%
%
\end{thmx}

Let us mention that the above theorem has a natural  fluid-dynamical interpretation. The physical context is the same as for Serrin's result, up to consider a {\em hollow} cylindrical pipe with a ring-shaped $2$-dimensional cross section $\Om$. In particular, our theorem says that {\em ``If the WSS is constant
on some connected component of the pipe's wall, then the velocity of the fluid may attain its maximal value only at finitely many streamlines, unless the hollow pipe itself consists of a couple of concentric cylindrical round tubes"}.


\subsection{A comparison algorithm.}
\label{sub:compa}


As it is clear from the previous discussion, the proof of Theorem~\ref{thm:const_grad} lies out of the range of applicability of some of the most powerful nowadays existing techniques. This fact has pushed us to develop a new approach, based on the introduction of a comparison algorithm. The comparison is meant to produce a number of sharp and rigid {\em a priori} bounds (e.g., Theorem~\ref{thm:min_pr_SD}, Proposition~\ref{pro:curv_bound}, and Proposition~\ref{pro:area_bound_Gamma}), eventually leading to the desired classification. More concretely, given a solution $(\Om, u)$  to problem~\eqref{eq:prob_SD}, we compare relevant analytic and geometric quantities (such as the gradient of $u$, or the lengths of the boundary components $|\Gamma_i|$ and $|\Gamma_o|$) with their corresponding counterparts on a ring-shaped model solution $(\Om_R, u_R)$ (namely with $|\nabla u_R|$, $|\Gamma_{R,i}| = 2 \pi r_i(R)$, and $|\Gamma_{R,o}| = 2 \pi$, to continue with the previous exemplification), bounding the former in terms of the latter. 

In order to obtain sharp and rigid inequalities, it is  crucial to take care of two important and intimately related aspects. The first one is the choice of an appropriate {\em scale fixing}. Indeed, since our problem is invariant by translations and rescaling (see~\eqref{eq:scaling}) and since we have already chosen a scale to describe the family of ring-shaped model solutions~\eqref{eq:ST}, it is convenient to normalise  in a consistent fashion also the generic solution $(\Om,u)$ that we aim to analyse. 
The second aspect is the selection of a good {\em basis for comparison}. In other words, since in the chosen normalisation each ring-shaped model solution is uniquely determined by the value of its {\em core radius}, we need to find a way to associate our generic solution with a number $0<R<1$. It actually turns out that the resolution of the latter problem also gives an answer to the scale fixing question. Indeed, once a core radius $0<R<1$ is selected, it is sufficient to rescale the solution according to~\eqref{eq:scaling} in such a way that  
\begin{equation*}
u_{\max} \,\, = \,\, (u_R)_{\max} \,.
\end{equation*}
Let us focus then on the problem of finding the most convenient value of the parameter $0<R<1$. The heuristic idea, here, is to use a shooting paradigm to guess the value of the core radius from the slope of the solution at the boundary, i.e., from the measurement of its wall shear stress. A closer look shows that a more refined information is actually needed.
Indeed, if two solutions are related to each other as in~\eqref{eq:scaling}, the value of the expected core radius must coincide. It is then convenient to replace the wall shear stress of a boundary component with its scaling invariant version. Moreover, it is clear that, in the outlined scheme, every single boundary component might in principle give rise to a different guess for the core radius parameter.
Let us take care of these two queries with a couple of definitions.

\begin{figure}
\centering
\includegraphics[scale=0.8]{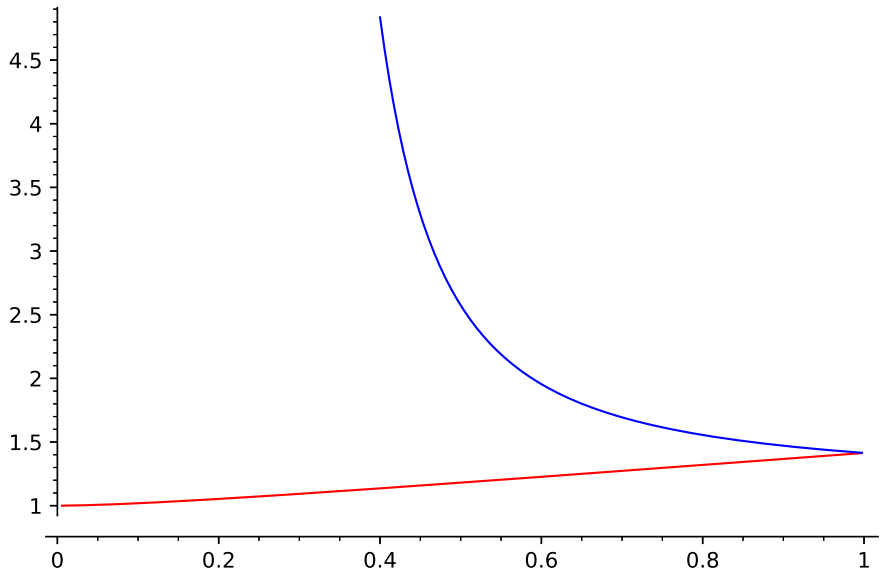}
\caption{The above figure represents the graphs of the normalised wall shear stress of the ring-shaped model solutions, as function of the core radius parameter $0<R<1$. In red we have the graph of the {\em outer NWSS function} $\overline\tau_o$, whereas in blue we have the graph of the {\em inner NWSS function} $\overline\tau_i$.
}
\label{fig:NWSS}
\end{figure}

\begin{definition}[Normalised Wall Shear Stress]
\label{def:NWSS}
Let $(\Om,u)$ be a solution to problem~\eqref{eq:prob_SD} and let $\Gamma \in \pi_0(\pa \Om)$ be a connected component of the boundary of $\Om$. We define the {\em normalised wall shear stress (NWSS)} of $\Gamma$ as
\begin{equation}
\label{eq:NWSS_BDARY}
\overline{\tau}(\Gamma)\,:=\,
\frac{\max_{\Gamma}|\D u|}{\sqrt{2\umax}} \, .
\end{equation}
More in general, if $N$ is a connected component of $\Om\setminus {\rm MAX}(u)$, we define the {\em normalised wall shear stress (NWSS)} of the region $N$ as 
\begin{equation}
\label{eq:NWSS_REG}
\overline{\tau}(N)\,:=\, \max\, \left\{ \, \overline{\tau} (\Gamma) \, : \, \Gamma \in \pi_0(\partial \Om \cap \overline{N}) \,  \right\} \, .
\end{equation}
\end{definition}

To introduce our second definition, it is important to observe that on ring-shaped model solutions the NWSS at either the inner or the outer boundary can be computed as a function of the core radius. More precisely, it is useful to consider the outer NWSS function $\overline\tau_o$ and the inner NWSS function $\overline\tau_i$, whose graphs are plotted in~\figurename~\ref{fig:NWSS} and which are defined as follows.
\begin{itemize}
\item 
The {\em outer NWSS function}  
\begin{equation}
\label{eq:k+}
\overline\tau_o  :  [\, 0, 1) \longrightarrow [\, 1, \sqrt 2\, )
\end{equation}
is defined by
\begin{equation}
\label{tau_o_2}
\hspace{-1.5cm}\overline\tau_o(R)\,:= \, \frac{|\D u_R|}{\sqrt{2(u_R)_{\rm max}}}\phantom{u\!\!}_{\big|_{\Gamma_{R,o}}}\!\!\!\! = \,\, 
\begin{dcases}
1,&\qquad\mbox{if }R=0,\\
\frac{1-R^2}{\sqrt{1-R^2+2R^2\log R}},& \qquad\mbox{if }0<R<1.
\end{dcases}
\end{equation}
Observe that $\overline\tau_o$ is continuous, strictly increasing, 
and $\overline\tau_o(R) \to\sqrt2$, as $R\to 1^-$.
\smallskip
\item 
The {\em inner NWSS function}   
\begin{equation}
\label{eq:k-}
\overline\tau_i  :  ( 0, 1 \,] \longrightarrow [\, \sqrt{2}, +\infty \, )
\end{equation}
is defined by
\begin{equation}
\label{tau_i_2}
\,\,\,\overline\tau_i(R)\,:= \, \frac{|\D u_R|}{\sqrt{2(u_R)_{\rm max}}}\phantom{u\!\!}_{\big|_{\Gamma_{R,i}}}\!\!\!\!\!\! = \,\,
\begin{dcases}
\frac{R^2-r_i^2(R)}{r_i(R)\sqrt{1-R^2+2R^2\log R}},& \qquad\mbox{if }0<R<1,\\
\sqrt 2,&\qquad\mbox{if }R=1.
\end{dcases}
\end{equation}
Observe that $\overline\tau_i$ is continuous, 
strictly decreasing, and $\overline\tau_i(R) \to + \infty$, as $R\to 0^+$.
\end{itemize}
As pointed out, the key feature of $\overline\tau_o$ and $\overline\tau_i$ is that they are invertible. Building on this property, we are now ready to introduce the notion of {\em expected core radius}. In analogy with the NWSS, this invariant can be associated to either a boundary component of $\pa \Om$
or, more in general, to a connected component of $\Om \setminus {\rm MAX}(u)$.
%
%
%
%
%
%
%
%
%
%
%
%
%



\begin{definition}[Expected Core Radius]
\label{def:virtual_mass} 
Let $(\Om,u)$ be a solution to problem~\eqref{eq:prob_SD} and let $\Gamma \in \pi_0(\pa \Om)$ be a connected component of the boundary of $\Om$. We define the {\em expected core radius} of $\Gamma$ as follows:
%
\begin{itemize}
\item[(i)] If $ 1 \leq\overline\tau(\Gamma)<\sqrt2$, we set
\begin{equation}
\label{virtual_R_1_gamma}
R(\Gamma)\,\, =\,\,\overline\tau_o^{-1}\big(\overline\tau(\Gamma)\big) \,,
\end{equation}
\item[(ii)] If $\overline\tau(\Gamma)\geq \sqrt{2}$, we set
\begin{equation}
\label{virtual_R_2_gamma}
R(\Gamma)\,\,  =\,\,\overline\tau_i^{-1}\big(\overline\tau(\Gamma)\big) \, ,
\end{equation}
\end{itemize}
More in general, if $N$ is a connected component of $\Om\setminus {\rm MAX}(u)$, we define the {\em expected core radius} of $N$ as follows:
\begin{itemize}
\item[(i)] If $\overline\tau(N)<\sqrt2$, we set
\begin{equation}
\label{virtual_R_1}
R(N)\,\, =\,\,\overline\tau_o^{-1}\big(\overline\tau(N)\big) \,.
\end{equation}
\item[(ii)] If $\overline\tau(N)\geq \sqrt{2}$, we set
\begin{equation}
\label{virtual_R_2}
R(N)\,\,  =\,\,\overline\tau_i^{-1}\big(\overline\tau(N)\big) \, .
\end{equation}
\end{itemize}
\end{definition}
As it is immediate to check, the expected core radius of a boundary component of a ring-shaped model solution $(\Om_R, u_R)$ coincides by construction with the value of its core radius parameter $R$. In other words we have
\begin{equation*}
R(\Gamma_{R,o}) \, = \, R \,= \, R(\Gamma_{R,i}) \, ,
\end{equation*}
and the same is true for the expected core radius of either the outer or the inner region $R(\Om_{R,o}) = R = R(\Om_{R,i})$. This picture also includes the extremal case of Serrin's solution~\eqref{eq:serrin_ST}, where the expected core radius is actually equal to $0$. 

It should be noticed that definition~\eqref{virtual_R_1_gamma} differs from definition~\eqref{virtual_R_1} in a subtle, though substantial way. In~\eqref{virtual_R_1_gamma} the condition $1 \leq \overline\tau (\Gamma)$ has to be imposed in order to get a number in the range of $\overline\tau_o$, that can be effectively used to define $R(\Gamma)$. Such a condition is not needed in~\eqref{virtual_R_1}, since it turns out to be always satisfied. In particular, the expected core radius $R(N)$ of a region $N \in \pi_0(\Om \setminus {\rm MAX}(u))$ is always well defined, as such, it is obviously nonnegative and more remarkably it vanishes if and only if $(\Om,u)$ is equivalent to Serrin's solution~\eqref{eq:serrin_ST}. This fact is stated in the following theorem, where no assumption is made {\em a priori} on either the topology of $\Om$, or the number its boundary components. 
\begin{thmx}
\label{thm:PMS}
Let $(\Om,u)$ be a solution to problem~\eqref{eq:prob_SD},
and let $N$ be a connected component of $\Om\setminus {\rm MAX}(u)$. Then, the expected core radius of $N$ is well defined and nonnegative 
$$
R(N)\,\geq\,0 \,.
$$
Moreover, the equality holds if and only if $(\Om,u)$ is equivalent to Serrin's solution~\eqref{eq:serrin_ST}.
\end{thmx}
The above theorem can be regarded as a first instance of how effective the notion of expected core radius might be for classification purposes. A second instance is contained in the following theorem. It says in particular that if the expected core radius that is guessed at the outer boundary $R(\Gamma_o)$ coincides with the one that is guessed at the inner boundary $R(\Gamma_i)$, then the solution is rotationally symmetric and coincides up to scaling with $(\Om_R, u_R)$, where $R(\Gamma_o) = R = R(\Gamma_i)$

\begin{thmx}
\label{thm:B}
Let $(\Om,u)$ be a solution to problem~\eqref{eq:prob_SD} such that $\Om$ is a ring-shaped domain and $u$ has infinitely many maximum points. Assume that $\overline\tau(\Gamma_o)<\sqrt{2}$. Then, the expected core radii of $\Gamma_o$ and $\Gamma_i$ are both well defined and positive. Moreover, they satisfy
$$
R(\Gamma_o)\,\geq\, R(\Gamma_i) \,>\, 0\, ,
$$
and the equality holds if and only if $(\Om,u)$ is equivalent to the ring-shaped model solution whose core radius is given by the common value of the  expected core radii. 
\end{thmx}
It is important to observe that, in contrast with Theorem~\ref{thm:const_grad}, no constant Neumann data are imposed on $\pa \Om$ in Theorem~\ref{thm:B}. Concerning the condition $\overline\tau(\Gamma_o) < \sqrt{2}$, it must be noticed that it is always satisfied on the model solutions, as ${\rm Im}(\overline\tau_o) = [1 , \sqrt{2})$. Building on this condition, we will deduce that 
$$
1 \leq \overline\tau(\Gamma_o) \qquad \hbox{and} \qquad \sqrt{2} \leq \overline\tau(\Gamma_i) \, .
$$ 
The first inequality implies that $R(\Gamma_o)$ is well defined, whereas the second says that the NWSS at $\Gamma_i$ lies in the range of $\overline\tau_i$, as it is natural to expect in the model situation.

\subsection{Comments and further directions.}

The results and techniques employed in this paper open the way to a number of natural questions and possibilities for further developments. 
In this subsection we select some of the ones that we consider more natural or stimulating.



Let us start with a basic, nonetheless fundamental, observation concerning our comparison algorithm.
As it is clear from Subsection~\ref{sub:compa}, in order to implement our method it is first crucial to select a rotationally symmetric solution to compare with. This is done by matching the NWSS of a region with the NWSS of the outer or inner region of a model solution.
A common feature of all the rotationally symmetric solutions $(\Omega_R,u_R)$ is that the NWSS of the inner region $\Omega_{R,i}$ is greater than $\sqrt{2}$, whereas the NWSS of the outer region $\Omega_{R,o}$ is less than $\sqrt{2}$. 
It would be interesting to figure out whether this holds true for a general ring-shaped solution to~\eqref{eq:prob_SD}, or whether there are counterexamples.
In a broader sense, one would want to understand to what extent the fact that a region is outer or inner influences the NWSS of that region. 
In the following sections we will be able to answer this question in some special cases: for instance, we will show that, if the NWSS is constant on the boundary of our region, then the region has low NWSS if it is outer and high NWSS if it is inner, as expected (see Proposition~\ref{pro:poho_ST}).
A complete answer to this question could potentially allow us to remove the hypothesis $\overline\tau(\Gamma_o)<\sqrt{2}$ from Theorem~\ref{thm:B}.



Another point that is worth commenting on is how the classical moving plane method compares with the method employed in this paper. As already mentioned in the discussion leading to Theorem~\ref{thm:Syr}, the moving plane method seems to be particularly effective in characterizing rotationally symmetric solutions that are monotonically decreasing along the radial coordinate. On the other hand, we will show that our comparison argument works well when dealing with monotonically increasing model solutions. It would then appear that the two methods are complementary.
Indeed, the moving plane method will come to our aid when discussing the rigidity statement of Theorem~\ref{thm:const_grad}, as we will need to invoke Theorem~\ref{thm:Syr} to deal with outer regions. 
It should however be stressed that, while the moving plane method has been extensively studied and its limitations are well understood (as well as its outstanding merits), our technique is essentially new, and the question is still open about whether and how far it is possible to extend its range of applicability. This point will be discussed in further details in the upcoming work~\cite{Bor_RS}, where it will be shown that our technique can be refined to deal with monotonically decreasing model solutions as well, albeit at some cost,  such as more subtle arguments and slightly less general conclusions.
It will be clear from~\cite{Bor_RS} that the results in the present paper can all be proved relying solely on our comparison approach. Nevertheless, in this paper we preferred to exploit Theorem~\ref{thm:Syr} in order to shorten and simplify the exposition.



Remaining on the topic of investigating the features of our technique, it would definitely be interesting to understand whether it can be employed to study other overdetermined problems. There are already some positive answers in this direction.
For instance, similar comparison techniques have been used in~\cite{Bor_Maz_2-I,Bor_Maz_2-II} to study static spacetimes with positive cosmological constant.
Also, as it will be shown in an upcoming work, our strategy allows to study problem~\eqref{eq:prob_SD} in higher dimension as well.
There are of course many other potential applications that deserve future investigation. For instance, a very natural direction would be the study of the problem $\De u=\delta_o$ inside a bounded domain containing the origin, since the rotationally symmetric solutions to this equation are very similar to the ones considered in this paper (namely, they are the same as~\eqref{eq:ST} but with negative values of $R^2$).
Another possible attempt would be to generalize the PDE considered in this paper, for instance considering the equation $\De u=f(u)$, for a suitable family of functions $f$.



It is also reasonable to expect that our comparison technique may find applications outside the realm of overdetermined boundary value problems. As a particularly relevant example, in~\cite{Mag_Pog_3} an estimate for the distance of the maximum points from the boundary is provided for solutions to~\eqref{eq:prob_SD}. This is achieved through a comparison argument with Serrin's solution~\eqref{eq:serrin_ST} by means of a classical elliptic inequality (see the proof of~\cite[Lemma~2.2]{Mag_Pog_3} for more details). One may wonder whether our comparison argument (more precisely the elliptic inequality and corresponding gradient estimate discussed in Subsection~\ref{sub:grad_est}) can be exploited to prove similar results.



We conclude by mentioning a couple of other problems coming from physical models, that are somehow related to ours. Remaining in the realm of fluid-dynamics, similar symmetry results on annular domains have been recently discussed for the Euler equation in~\cite{Ham_Nad}. 
Another problem that is worth mentioning is the so called torsion problem, modeling the torsion of a bar with holes.  
 From a mathematical viewpoint, this problem is very similar to the one discussed here, the only difference being the boundary condition in the inner boundary components. 
One of the main papers on this problem is the well known~\cite{Pol_Wei}, where Schwarz symmetrization~\cite{Schwarz} is used to prove that the rotationally symmetric solutions to the torsion problem are characterized as the ones maximizing torsional rigidity. Among the recent progresses on this problem, we mention~\cite{Pay_Phi}, based again on the moving plane method, and~\cite{Cra_Fra_Gaz}.

\subsection{Plan of the paper}

In the rest of the paper we will prove the results stated in this section.
Section~\ref{sec:comparison_geometry} is dedicated to the proof
of Theorem~\ref{thm:PMS}. This result is crucial as it grants us that Definition~\ref{def:virtual_mass} is well posed, which in turn allows us to set up our comparison machinery. This is the topic of Section~\ref{sec:grad_est}, where the important notion of pseudo-radial function is introduced and exploited to prove a sharp upper bound for
the gradient of $u$. 
In Section~\ref{sec:theo_B} we will  
see that the gradient estimates lead to some 
curvature bounds for $\Gamma_o$, $\Gamma_i$, and ultimately to the proof of Theorem~\ref{thm:B}.
Such curvature bounds will also be used in Section~\ref{sec:theo_A},
suitably coupled with the Pohozaev Identity,
to obtain Theorem~\ref{thm:const_grad}.
Finally, Section~\ref{Sec:bif} is devoted to the proof of
Theorem~\ref{thm:bifurcation}.

\section{Proof of Theorem~\ref{thm:PMS}: the expected core radius of a region}
\label{sec:comparison_geometry}

This section is devoted to the proof of the following theorem, which implies at once Theorem~\ref{thm:PMS}, and in turns the fact that the expected core radius of a region $N \in \pi_0(M \setminus {\rm MAX}(u))$ is well defined, nonnegative and vanishes if and only if $(\Om, u)$ is equivalent to a Serrin's solution.

\begin{theorem}
\label{thm:main_ST}
Let $(\Om, u)$ be a solution to problem~\ref{eq:prob_SD}, where $\Om$ is an arbitrary bounded domain $\Om$, with smooth boundary.
Let $N$ be a 
connected component of $\Om\setminus{\rm MAX}(u)$ and suppose that $\overline\tau (N) \leq 1$, i.e.
$$
\max_{\pa\Om \cap \overline{N}}\frac{ |\D u|}{\sqrt{2\umax}} \, \leq \, 1 \, .
$$
Then, up to suitable dilations and translations, 
$(\Om,u)$ coincides with Serrin's solution~\eqref{eq:serrin_ST}.
\end{theorem}

\noindent
The first step in the proof is to use a Maximum Principle argument to establish the following weaker version of Theorem~\ref{thm:main_ST}.

\begin{proposition}
\label{thm:shen_ST}
Let $(\Om, u)$ be a solution to problem~\ref{eq:prob_SD}, where $\Om$ is an arbitrary bounded domain $\Om$, with smooth boundary, and suppose that 
$$
\max_{\pa\Om}\frac{ |\D u|}{\sqrt{2\umax}} \, \leq \, 1 \, .
$$
Then, up to suitable dilations and translations, 
$(\Om,u)$ coincides with Serrin's solution~\eqref{eq:serrin_ST}.
\end{proposition}

\begin{proof}
The Bochner formula coupled with the first equation in~\eqref{eq:prob_SD} gives
\begin{equation}
\label{eq:bochner_maxpr_ST}
\De\big(|\D u|^2\,+\,2u\big)\,=\,2\,\Big[\,|\DD u|^2-\frac{(\De u)^2}{2}\,\Big]\,\geq\, 0\,.
\end{equation}
Since $|\D u|^2\leq 2\umax$ on $\pa\Om$ by hypothesis, 
the Maximum Principle implies that
$$
\max_{\Om}\Big(|\D u|^2+2u\Big)\,=\,\max_{\pa\Om}|\D u|^2\,\leq\,2\umax\,.
$$
On the other hand, at any maximum point for $u$ in $\Omega$
it holds $|\D u|^2+2 u=2\umax$.
Hence, by the Strong Maximum Principle we obtain that $|\D u|^2+2 u \equiv 2\umax$ in $\Om$.
In turn, the equality holds in~\eqref{eq:bochner_maxpr_ST},
which yields $\na^2u=-g_{\R^2}$ and hence the conclusion.
\end{proof}

Having fixed a connected component $N$ of $\Om \setminus {\rm MAX}(u)$, we now introduce the function $U:[0,\umax)\rightarrow\R$ given by
\begin{equation}
\label{eq:Up_ST}
t \,\, \longmapsto \,\, U(t) 
\, =\, 
\frac1{(\umax-t)}  \int\limits_{ \{ u = t \}\cap\overline N}
\!\!\!\!  |\D u|\, \rmd \sigma.
\end{equation}
This function is well defined, because the integrand is globally bounded and 
classical results ensure that the level sets of $u$ have finite $\mathscr{H}^1$-measure.

The second step in the proof of Theorem~\eqref{thm:main_ST}
consists in showing that the function $U$ is 
nonincreasing.

\begin{proposition}
\label{pro:main_ST}
Let $(\Om, u)$ be a solution to problem~\ref{eq:prob_SD}, where $\Om$ is an arbitrary bounded domain $\Om$, with smooth boundary.
Let $N$ be a 
connected component of $\Om\setminus{\rm MAX}(u)$ and suppose that $\overline\tau (N) \leq 1$, i.e.
$$
\max_{\pa\Om \cap \overline{N}}\frac{ |\D u|}{\sqrt{2\umax}} \, \leq \, 1 \, .
$$
Then the function $U$ defined in~\eqref{eq:Up_ST} is continuous and nonincreasing. 
\end{proposition}

\begin{proof} Given a region $N \in \pi_0(\Om \setminus {\rm MAX}(u))$, let us consider the domain $N_\eta=N\cap \{ u < \umax -\eta\}$, where $\eta\in \R$ is small enough so that the level set $\{u=\eta \}$ is regular. Notice that, since the critical level sets of $u$ are discrete (see~\cite{Sou_Sou_1}), the parameter $\eta$ introduced above can be chosen as small as desired. Applying the Maximum Principle to~\eqref{eq:bochner_maxpr_ST}, we obtain that 
$$
\max_{N_\eta} \, \big( |\na u|^2 + 2 u \big) \, \leq \, \max_{\pa N_\eta} \, \big(|\na u|^2 + 2 u \big)
$$
On the other hand one has that
$$
\lim_{\eta\to 0^+}\max_{\pa N_\eta}(|\D u|^2+2u)\,\leq\, 2\umax\,.
$$
In fact, $|\D u|^2+2u=|\D u|^2\leq 2\umax$ on $\pa \Om \cap \overline{N}$ by hypothesis and
$$
\lim_{\eta\to 0^+} \max_{\{u = \umax - \eta \}} (|\D u|^2+2u) \, = \, 2 \umax \,,
$$
since $|\D u|\to 0$ as we approach ${\rm MAX}(u)$. It follows that
$$
|\D u|^2+2u\,\leq\,2\umax
$$
on the whole $N$. In particular, using that $\De u=-2$, we obtain
\begin{equation}
\label{eq:div_Du_ST}
{\rm div}\left(\frac{\D u}{\umax-u}\right)
\,=\,
\frac{|\D u|^2+2u-2\umax}{(\umax-u)^2}
\,\leq\,0\,.
\end{equation} 
We now integrate by parts inequality~\eqref{eq:div_Du_ST} on the finite perimeter set $\{t_1\leq u\leq t_2\}\cap\overline N$, for some $0\leq t_1<t_2<\umax$. Applying the divergence theorem, we deduce that
\begin{multline}
\label{eq:int_div_Du_ST}
\int\limits_{\{u=t_1\}\cap\overline N}
\!\!\!\!\bigg\langle\frac{\D u}{\umax-u}\,\bigg|\,{\rm n}\bigg\rangle\,\rmd\sigma
\,+
\!\!\!\int\limits_{\{u=t_2\}\cap\overline N}\!\!\!\!\bigg\langle\frac{\D u}{\umax-u}\,\bigg|\,{\rm n}\bigg\rangle\,\rmd\sigma\\
\,=\,
\int\limits_{\{t_1\leq u\leq  t_2\}\cap\overline N}\!\!\!\!\!
\frac{|\D u|^2+2u-2\umax}{(\umax-u)^2}
\,\,\rmd\mu\,\leq\,0\,.
\end{multline}
Notice that the (measure theoretic) unit normal ${\rm n}$ is well defined $\mathscr{H}^1$-a.e. on $\{u = t_1\}$ and on $\{u=t_2\}$. The thesis follows noticing that
\begin{equation*}
U(t) 
\,\, =\,\, 
\frac1{(\umax-t)}  \!\!\!\!\!\!\! \int\limits_{ (\{ u = t \}\cap\overline N) \setminus {\rm Crit}(u)}
\!\!\!\!\!\!\!\!  |\D u|\, \rmd \sigma \,\,\, = \!\!\!\!\!\!\!\!\!\!\!\!\!\!\!  \int\limits_{(\{u=t\}\cap\overline N) \setminus {\rm Crit}(u)}
\!\!\!\!\!\!\!\bigg\langle\frac{\D u}{\umax-u}\,\bigg|\, \pm{\rm n}\bigg\rangle\,\rmd\sigma \,\, =  \!\!\!\!\!\!\! \int\limits_{(\{u=t\}\cap\overline N)}
\!\!\!\!\bigg\langle\frac{\D u}{\umax-u}\,\bigg|\, \pm{\rm n}\bigg\rangle\,\rmd\sigma \, ,
\end{equation*}
since the outer unit normal ${\rm n}$ coincides with $\pm\na u/|\na u|$ on $(\{u=t\}\cap\overline N) \setminus {\rm Crit}(u)$. The continuity of $U$ is a straightforward consequence of the absolute continuity 
of the Lebesgue integral on the right-hand side of the equality in~\eqref{eq:int_div_Du_ST}.
\end{proof}

Combining Propositions~\ref{thm:shen_ST} and~\ref{pro:main_ST}, we are finally able to prove the main result of this section.

\begin{proof}[Proof of Theorem~\ref{thm:main_ST}.]
We claim that if $\overline\tau (N) \leq 1$, then $\pi_0(\Om \setminus {\rm MAX}(u)) = \{N\}$. In other words, $N$ is the only connected component of $\Om \setminus {\rm MAX}(u)$, under our assumptions. Observe that such a claim implies that $\pa \Om \cap \overline{N} = \pa \Om$, and thus the thesis follows from Proposition~\ref{thm:shen_ST}. To prove the claim, we argue by contradiction. If $N$ is not the unique connected component of $\Om \setminus {\rm MAX}(u)$, then it must be separated by the other ones, and thus $\mathscr{H}^1\big({\rm MAX}(u)\cap \overline{N}\big)> 0$.
On the other hand, using the classical {\L}ojasiewicz Inequality (see, e.g.,~\cite{Lojasiewicz_1}) and the  compactness of ${\rm MAX}(u)$, there exists a neighborhood $V$ of ${\rm MAX}(u)$ and two constants $c>0$ and $1/2\leq\theta<1$ such that
$\umax-u<1$ in $V$ and 
\[
|\na u(x)|
\,\geq\,
c\,(\umax-u(x))^\theta,
\]
for every $x\in V$. In particular, since $\{u=t\}\cap\overline N\subseteq V$
for every $t$ sufficiently close to $\umax$, we have that
\begin{equation*}
\frac{1}{\umax}\, \int\limits_{\pa \Om \cap \overline{N}} \!  |\D u| \, \rmd \sigma \, = \, 
U(0)
\, \geq \,
U(t)
\,\geq\,
\frac1{(\umax-t)^{1-\theta}}\,\big|\{u=t\}\cap\overline N\big| \, ,
\end{equation*}
where we also used the monotonicity formula proven in Proposition~\ref{pro:main_ST}. It is now easy to see that if $\mathscr{H}^1\big({\rm MAX}(u)\cap \overline{N}\big)> 0$, the rightmost hand side is unbounded, as $t \to \umax^-$. This gives the desired contradiction. 
\end{proof}
For the sake of completeness we now show how Theorem~\ref{thm:PMS} can be easily deduced from Theorem~\ref{thm:main_ST}. 
\begin{proof}[Proof of Theorem~\ref{thm:PMS}]
To check that $R(N)$ is well defined, it is sufficient to exclude that $\overline\tau(N)<1$, for some $N \in \pi_0(\Om \setminus {\rm MAX}(u))$. Indeed, Theorem~\ref{thm:main_ST} says that if $\overline\tau (N)\leq 1$, then $\overline\tau(N)=1$. To prove the rigidity part, it is sufficient to observe that if $R(N)=0$, then $\overline\tau(N)=1$, and one can apply again Theorem~\ref{thm:main_ST}.
\end{proof}


\section{Gradient estimates}
\label{sec:grad_est}


The aim of this section is to compare the gradient $|\na u|$ of a generic solution $(\Om, u)$ to problem~\ref{eq:prob_SD}
with the gradient $|\na u_R|$ of a ring-shaped model solution $(\Om_R, u_R)$. In order for the comparison to make sense, we need to consider suitable normalisations, as described in Subsection~\ref{sub:compa}. A crucial role in this procedure will be played by the concept of {\em expected core radius} $R=R(N)$ of a region $N \subseteq \Om \setminus {\rm MAX}(u)$, whose existence is now guaranteed by Thoerem~\ref{thm:PMS}. To be more precise, we notice that, if $(\Om,u)$ is a solution to problem~\eqref{eq:prob_SD}, then $(\Omega_\lambda,u_\lambda)$, with 
\begin{equation}
\label{eq:scaling}
\Omega_\lambda\,=\,\{\lambda \,x\,:\,x\in\Omega\}\, \quad\hbox{ and }\quad u_\lambda(x)= {\lambda^2}\,u(x/\lambda)\,,
\end{equation}
is also a solution for every $\lambda> 0$. This means that we are allowed to rescale $u$, provided that we also apply a suitable homothety to the domain $\Omega$. With the notation introduced in~\eqref{eq:umax_ST}, it will be convenient to adopt the following normalisation.

\begin{normalization}
\label{norm}
Let $(\Om, u)$ be a solution to problem~\eqref{eq:prob_SD}, let $N$ be a connected component of $\Om\setminus{\rm MAX}(u)$, and let $R = R(N) \in [0,1)$ be the expected core radius associated with the region $N$. Up to rescale the domain and the function as in~\eqref{eq:scaling}, we assume that 
\begin{equation}
\label{eq:rescaling}
\umax \, = \, (u_R)_{\max} \,=\,
\frac{1-R^2}{2}+R^2\log{R} \, .
\end{equation}
\end{normalization}

With this normalisation in force, we are going to compare, in Theorem~\ref{thm:min_pr_SD} below, the squared gradient of the solution $|\D u|^2$ on $N$ to the squared gradient $|\na u_R|^2$ of the ring-shaped model solution $(\Om_R, u_R)$ that satisfies $R=R(N)$. We agree that if $\overline\tau(N) < \sqrt{2}$, then the comparison will be drawn with the restriction of $|\na u_R|^2$ to the outer region $\Om_{R,o}$ of the model solution, otherwise the comparison will be drawn with the restriction of $|\na u_R|^2$ to the inner region $\Om_{R,i}$ (see~\eqref{eq:Om+-_ST}). More precisely, if $p$ is a point in $N$, we are going to bound $|\na u|^2(p)$ with the value of $|\na u_R|^2$ at a point where $u_R = u(p)$, belonging to either the outer or the inner region of the model solution $(\Om_R,u_R)$, according to what the NWSS of $N$ dictates. To make the computations effortable, we are going to introduce, in the next subsection, the notion of {\em pseudo-radial function}.

\subsection{The pseudo-radial functions}
\label{sub:pseudo_radial}

This subsection is aimed at defining {\em pseudo-radial functions}, that is, a functions that mimic the behavior of the radial coordinate $|x|$ in the rotationally symmetric solutions~\eqref{eq:ST}. As above, let $N \in \pi_0(\Om \setminus{\rm MAX}(u))$ and let $R=R(N)$ be its expected core radius. As in~\eqref{eq:ST}, we let $r_i = r_i(R) \in (0,R)$ be the smallest positive root of the function $ \rho \mapsto 1-\rho^2+2R^2\log \rho$, and we define the function
\begin{align*}
F_R:[0, (u_R)_{\max}]\times [r_i(R),1]\,\,&\longrightarrow\,\, \R
\\
\phantom{\qquad} (u,\psi)\,\,\longmapsto\,\, F_R(u,\psi)\,=\,2u-1&+\psi^2-2R^2\log\psi
\end{align*}
A simple computation shows that $\pa F_R/\pa\psi=0$ if and only if $\psi=R$. As a consequence of the Implicit Function Theorem we have that there exist two smooth functions

\begin{equation*}
\psi_-:[0,(u_R)_{\max}]\,\,\longrightarrow\,\,\left[r_i(R),R\right]\qquad \hbox{and}\qquad
\psi_+:[0,(u_R)_{\max}]\,\,\longrightarrow\,\,\left[R,1\right]\,,
\end{equation*}
such that $F_R(u,\psi_-(u))= 0 = F_R(u,\psi_+(u))$ for all $u\in[0,\umax(m)]$.
For future convenience, let us list some elementary properties of $\psi_+$ and $\psi_-$, that can be derived easily from their definition. 
\begin{itemize}
\item First of all, we can compute $\psi_+$, $\psi_-$ and their derivatives using the following formul\ae
\begin{equation}
\label{eq:def_Psi_SD}
u\,=\,\frac{1-\psi_{\pm}^2+2R^2\log\psi_\pm}{2}\,.
\end{equation}
\begin{equation}
\label{eq:dpsideu_SD}
\dot{\psi}_{\pm}\,=\,-\frac{\psi_{\pm}}{\psi_{\pm}^2- R^2}\,,\qquad \ddot{\psi}_{\pm}\,=\,2\,\dot{\psi}_{\pm}^3+\frac{\dot{\psi}_{\pm}^2}{{\psi}_{\pm}}\,.
\end{equation}

\item The function $\psi_-$ takes values in $[r_i(R),R]$, hence $\psi_-^2\leq R^2$ and from~\eqref{eq:dpsideu_SD} we deduce
$$
\dot\psi_-\,\geq\,0\,,\qquad \ddot\psi_-\,\geq\, 0\,,\qquad \lim_{u\to (u_R)_{\max}^-}\dot\psi_-\,=\,+\infty\,.
$$

\item The function $\psi_+$ takes values in $[R,1]$, hence $\psi_+^2\geq R^2$ and from the first formula in~\eqref{eq:dpsideu_SD} we deduce  that $\dot\psi_+$ is nonpositive and diverges as $u$ approaches $(u_R)_{\max}$. Moreover, the second formula in~\eqref{eq:dpsideu_SD} can be rewritten as
$$
\ddot\psi_+\,=\,\dot\psi_+^3\left(1+ R^2\psi_+^{-2}\right)\,,
$$ 
from which it follows $\ddot\psi_+\leq 0$.
Summing up, we have
$$
\dot\psi_+\,\leq\,0\,,\qquad \ddot\psi_+\leq 0\,,\qquad \lim_{u\to (u_R)_{\max}^-}\dot\psi_+\,=\,-\infty\,.
$$
\end{itemize}

Coming back to our case of interest, we are now going to use the functions $\psi_\pm$ in order to define a {\em pseudo-radial function} on $N$. To this end, we distinguish the cases where the NWSS OF $N$ is either above or below the threshold value $\sqrt{2}$.

\begin{definition}[Pseudo-radial functions]
\label{def:psudo}
Let $(\Om, u)$ be a solution to problem~\eqref{eq:prob_SD}, let $N$ be a connected component of $\Om\setminus{\rm MAX}(u)$, and let $R = R(N) \in [0,1)$ be the expected core radius associated with the region $N$. Also assume that the Normalisation~\ref{norm} is in force.
\begin{itemize}
\item[$(i)$] If $\overline\tau(N)<\sqrt{2}$, then we define the pseudo-radial function $\Psi_+$ as
\begin{equation}
\label{eq:pr_function_+}
\begin{split}
\Psi_+:\,N &\,\longrightarrow\, [R,1]
\\
p&\,\longmapsto\, \Psi_+(p):=\psi_+(u(p))\,. 
\end{split}
\end{equation}
Notice that, if $N$ is the outer region $\Om_{R,o}$ of the rotationally symmetric solution~\eqref{eq:ST} with core radius $R$, then, for every $p\in \Om_{R,o}$, the value of $\Psi_+(p) = \psi_+ (u_R(p))$ is equal to the value of the radial coordinate $|x|$ at $p$.
\smallskip
\item[$(ii)$] If $\overline\tau(N)>\sqrt{2}$, then we define the pseudo-radial function $\Psi_-$ as
\begin{equation}
\label{eq:pr_function_-}
\begin{split}
\Psi_-:\,N &\,\longrightarrow\, [r_i(R),R]
\\
p&\,\longmapsto\, \Psi_-(p):=\psi_-(u(p))\,. 
\end{split}
\end{equation}
Notice that, if $N$ is the inner region $\Om_{R,i}$ of the rotationally symmetric solution~\eqref{eq:ST} with core radius $R$, then, for every $p\in \Om_{R,i}$, the value of $\Psi_-(p) = \psi_- (u_R(p))$ is equal to the value of the radial coordinate $|x|$ at $p$.
\end{itemize}
\end{definition}

\begin{remark}
The threshold case $\overline\tau(N)=\sqrt{2}$ is not considered in the above definition. In fact, for that value one has $r_i(R)=R=1$, so either~\eqref{eq:pr_function_+} or~\eqref{eq:pr_function_-} would give us a pseudo-radial function that is just constant on the whole $N$, and as such not interesting. The reason for this issue should be traced back to the fact that no rotationally symmetric solution $(\Omega_R,u_R)$ has a boundary component with NWSS equal to $\sqrt{2}$, hence when $\overline\tau(N)=\sqrt{2}$ we do not have a a model to compare with. For this reason, in the future
our analysis will be mostly focused on the cases $\overline\tau(N)<\sqrt{2}$ and $\overline\tau(N)>\sqrt{2}$, whereas the case $\overline\tau(N)=\sqrt{2}$ will 
be treated separately with ad hoc arguments.
\end{remark}

In analogy with~\eqref{eq:dpsideu_SD} and for future convenience, we point out that the following relationships hold true between the derivatives of the pseudo-radial function $\Psi$ and the potential $u$.
\begin{equation}
\label{eq:rel_u_gradient_Psi}
\D \Psi_{\pm}\,=\,(\dot{\psi}_\pm\circ u)\,\D u\,,\qquad\DD \Psi_\pm\,=\,(\dot{\psi}_\pm\circ u)\,\DD u\,+\,(\ddot{\psi}_\pm\circ u)\,du\otimes du\,.
\end{equation}

\begin{notation}
In the following sections, we will perform several formal computations. In order to simplify the notations, we will avoid to indicate the subscript $\pm$, and we will simply denote by $\Psi=\psi\circ u$ the pseudo-radial function on a region $N$ of $\Om\setminus{\rm MAX}(u)$, where we understand that $\Psi$ is defined by~\eqref{eq:pr_function_+} if we are in an outer region and by~\eqref{eq:pr_function_-} if we are in an inner region. When there is no risk of confusion, we will also avoid to explicitate the composition with $u$, namely, we will write $\psi, \dot\psi$ or $\ddot\psi$ instead of $\psi\circ u, \dot\psi \circ u$ or $\ddot\psi \circ u$, respectively. For instance, the formul\ae\ in~\eqref{eq:rel_u_gradient_Psi} will be simply written as
$$
\D \Psi\,=\,\dot{\psi}\,\D u\,,\qquad\DD \Psi\,=\,\dot{\psi}\,\DD u\,+\,\ddot{\psi}\,du\otimes du\,,
$$
\end{notation}
 From the definition of $\Psi$, given in~\eqref{eq:pr_function_+},~\eqref{eq:pr_function_-}, we can also easily recover an explicit formula for the comparison function $W_R = |\na u_R|^2 \circ \Psi$ as a function of the pseudo-radial function:
\begin{equation}
\label{eq:Wm}
W_R\,=\,\left(\frac{\Psi^2-R^2}{\Psi}\right)^{\!\!2}.
\end{equation}


\subsection{Gradient estimates}
\label{sub:grad_est}


We are now ready to state the main result of this section, in which we prove that the function $W = |\na u|^2$ is bounded from above by $W_R= |\na u_R|^2 \circ \Psi$, where $R=R(N)$ is the expected core radius of the region $N$, that we are considering.

\begin{theorem}[Gradient Estimates]
\label{thm:min_pr_SD}
Let $(\Om,u)$ be a solution to problem~\eqref{eq:prob_SD}, let $N$ be a connected component of $\Omega\setminus{\rm MAX}(u)$,  and let $R = R(N) \in [0,1)$ be the expected core radius associated with the region $N$. Also assume that  Normalization~\ref{norm} is in force. 
Then it holds 
$$
W \, \leq \,  W_R\,,\quad \hbox{i.e.} \quad|\D u| \, \leq \, |\D u_R| \circ \Psi \,,
$$ 
on the whole $N$. Moreover, if $W=W_R$ at some point in the interior of $N$, then $(\Omega,u)$ coincides with the ring-shaped model solution with core radius $R$. 
\end{theorem}

\begin{proof}
We start by using the Bochner formula to compute the following:
\begin{align}
\notag
\De (W-W_R)\,&=\,\De|\na u|^2\,+\,2\left(1+\frac{R^2}{\Psi^2}\right)\De u\,-\,\frac{4\,R^2}{\Psi^2(R^2-\Psi^2)}|\na u|^2
\\
\label{eq:De_WWo}
&=\,2\,|\nana u|^2\,-\,4\left(1+\frac{R^2}{\Psi^2}\right)\,-\,\frac{4\,R^2}{\Psi^2(R^2-\Psi^2)}|\na u|^2
\end{align}

The next step is to find a suitable estimate for $|\nana u|^2$. It turns out that the one we need is obtained from the following quantity:
$$
\left|\,\nana u\,+\,\frac{2R^2}{(R^2-\Psi^2)^2}du\otimes du\,+\left(1\,-\,\frac{R^2}{(R^2-\Psi^2)^2}|\D u|^2\right)g_{\R^2}\,\right|^2\,\geq\,0\,.
$$
Computing explicitly the squared norm and isolating the term $|\nana u|^2$, we obtain the estimate
\begin{equation*}
|\nana u|^2\,\geq\,-\,\frac{2\,R^2}{(R^2-\Psi^2)^2}\langle\na(W-W_R)|\na u\rangle\,
\\
+\,2\,\left[1\,+\,\frac{2\,R^4}{\Psi^2(R^2-\Psi^2)^2}|\D u|^2\,-\,\frac{R^4}{(R^2-\Psi^2)^4}|\na u|^4\right]\,.
\end{equation*}

We can see why this estimate is the appropriate one for our purposes by plugging it in~\eqref{eq:De_WWo}, as by doing so we obtain an elliptic inequality for the quantity $W-W_R$:

\begin{multline*}
\De(W-W_R)\,\geq\,-\,\frac{4\,R^2}{(R^2-\Psi^2)^2}\langle\na(W-W_R)|\na u\rangle\,
\\
+\,\frac{4\,R^2}{(R^2-\Psi^2)^2}\,\left[1\,-\,\frac{R^2}{(R^2-\Psi^2)^2}|\D u|^2\right](W-W_R)
\end{multline*}
Unfortunately this is not yet enough, as we do not have a sign for the coefficient of the zero order term. For this reason, we then consider the new function $F_\beta=\beta\,(W-W_R)$, where $\beta=\beta(\Psi)>0$. A simple computation gives
\begin{multline*}
\De F_\beta\,\geq\,-\,\frac{2\Psi}{R^2-\Psi^2}\left[\frac{\beta'}{\beta}\,-\,\frac{2\,R^2}{\Psi(R^2-\Psi^2)}\right]\langle\na F_\beta|\na u\rangle\,-\,\frac{2\,\Psi}{R^2-\Psi^2}\,\left[\frac{\beta'}{\beta}\,-\,\frac{2\,R^2}{\Psi(R^2-\Psi^2)}\right]\,F_\beta
\\
+\,\frac{\Psi^2\,|\D u|^2}{(R^2-\Psi^2)^2}\,\left[\left(\frac{\beta'}{\beta}\right)'\,-\,\left(\frac{\beta'}{\beta}\right)^2\,+\,\frac{\Psi^2+5R^2}{\Psi(R^2-\Psi^2)}\frac{\beta'}{\beta}\,-\,\frac{4\,R^4}{\Psi^2(R^2-\Psi^2)^2}\right]F_\beta\,,
\end{multline*}
where we have  used $'$ to denote the differentiation with respect to $\Psi$.
We now need to find a function $\beta$ such that the coefficients  of the zero order terms have the right sign.
A good choice is to set
$$
\frac{\beta'}{\beta}\,=\,\frac{2R^2}{\Psi(R^2-\Psi^2)}\,,
$$
which corresponds to choosing
$$
\beta\,=\,\frac{\Psi}{\sqrt{W_R}}\,=\,\frac{\Psi^2}{|R^2-\Psi^2|}\,.
$$
With this choice, $F_\beta$ satisfies
$$
\De F_\beta\,-\,\frac{8\,R^2 \, \Psi^2}{(R^2-\Psi^2)^4}\,|\D u|^2\,F_\beta\,\geq\,0\,.
$$
It follows that $F_\beta$ satisfies the Maximum Principle in $N$. Since $W\leq W_R$ on the horizon with maximum surface gravity, we have
$$
F_\beta\leq 0 \hbox{ on $\Gamma_N$}.
$$
To see the behaviour of $F_\beta$ near ${\rm MAX}(u)$ we rewrite that quantity as
$$
F_\beta\,=\,\beta\,(W-W_R)\,=\,\Psi\,\frac{W}{\sqrt{W_R}}\,-\,\Psi\,\sqrt{W_R}\,.
$$ 
Notice that $W$ and $W_R$ go to zero as we approach ${\rm MAX}(u)$. Furthermore, using the expansion proven in Lemma~\ref{le:Wm_umax} below, we have
that for every $p\in{\rm MAX}(u)\cap\overline{N}$ 
$$
\lim_{x\to p,\,x\in N}\frac{W}{\sqrt{W_R}}\,=\,\lim_{x\to p,\,x\in N}\frac{|\D u|^2}{2\,\sqrt{\umax-u}}
$$
Now we apply the Reverse {\L}ojasiewicz Inequality~\cite[Theorem~2.2]{Bor_Chr_Maz}, to conclude that the limit on the right hand side is zero.
Therefore, $F_\beta$ tends to zero as we approach ${\rm MAX}(u)$. The Maximum Principle then implies that $F_\beta\leq 0$ (equivalently,
$W\leq W_R$) on the whole $N$. Furthermore, if the equality $W=W_R$ holds at one point $p$ in the interior of $N$, then, applying the Strong Maximum Principle in a neighborhood of $p$, we deduce that $W=W_R$ on the whole $N$, providing the desired rigidity statement. 
\end{proof}


\section{Proof of Theorem~\ref{thm:B}: curvature bounds and comparison geometry}
\label{sec:theo_B}


In this section we are going to exploit the gradient estimates proven in Theorem~\ref{thm:min_pr_SD} to deduce in Proposition~\ref{pro:curv_bound} some geometric {\em a priori} bounds on the curvature of the boundary components of $\Om$. Analogous results are then obtained in Proposition~\ref{pro:curv_bound_Sigma} for the curvature of the top stratum of ${\rm MAX}(u)$, whenever it is present. Building on this latter, we will then give a proof of Theorem~\ref{thm:B}. We start with the curvature bounds that are taking place at the boundary components of $\Om$. As clarified in Section~\ref{sec:theo_A}, these will play a crucial role in the proof of Theorem~\ref{thm:const_grad}.

\begin{proposition}
\label{pro:curv_bound}
Let $(\Om,u)$ be a solution to problem~\eqref{eq:prob_SD}, let $N$ be a connected component of $\Omega\setminus{\rm MAX}(u)$,  and let $R = R(N) \in [0,1)$ be the expected core radius associated with the region $N$. Also assume that  Normalization~\ref{norm} is in force. Then, at any point $p \in \pa \Om \cap \overline{N}$ where 
\begin{equation*}
|\na u| (p) \,  = \, \max_{\pa \Om \cap \overline{N}} |\na u| \, ,
\end{equation*}
it holds
\begin{equation}
\label{eq:curv_bound}
\kappa(p) \leq 1\,, \quad \hbox{if $\overline\tau(N)<\sqrt{2}$} \qquad \hbox{and} \qquad \kappa(p) \leq -\frac{1}{r_i(R)}\,, \quad \hbox{if $\overline\tau(N) > \sqrt{2}$} \, ,
\end{equation}
where $\kappa (p)$ is the curvature of $\pa \Om$ at $p$, computed with respect to the exterior unit normal.
\end{proposition}

\begin{proof}
Let $p \in \pa\Om \cap \overline{N}$ be a point as in the statement and let
 $\gamma(s) = p + (\na u/|\na u|) (p) \, s$. In other words, $\gamma$ is a unit speed straight segment starting at $p$ and pointing towards the interior of $N$. 
If $W$ and $W_R$ are the functions defined as in the {\em incipit} of Subsection~\ref{sub:grad_est}, it is readily checked that $W(p)=W_R(p)$. The Taylor expansion of $W$ along $\gamma$ gives
$$
W(\gamma(s))\,=\,W (p)\,-\,2 \Big[ \,\big(2\,-\,\kappa(p)\,\sqrt{W (p)}\,\big)\sqrt{W(p)} \,\Big]\,s\,+\,o(s)\,,
$$
where we used the identity
$$
\kappa(p)\,=\,\frac{\DD u_{|_p}({\D u},{\D u})\,-\,|\D u|^2 \De u(p)}{|\D u|^3(p)}\,=\,\frac{2 |\D u|^2\,+\,\DD u_{|_p}({\na u},{\na u})}{|\D u|^3(p)}\,,
$$
the curvature $\kappa$ being computed with respect to the exterior unit normal $-\D u / |\D u|$.
To obtain the expansion of $W_R$ along $\gamma$ it is convenient to make use of~\eqref{eq:Wm}. This leads to
\begin{align*}
W_R(\gamma(s))\,&=\,W_R(p)\,+\,\bigg\langle\D W_R(p)\,\bigg|\,\frac{\D u}{|\D u|}(p)\bigg\rangle\,s\,+\,o(s)
\\
&=\,W_R(p)\,+\,\left[\frac{\pa}{\pa\Psi} \!\left(\frac{\Psi^2-R^2}{\Psi}\right)^{\!\!2}\!\!(p)\, \,\, \dot\psi(0)\,\,|\D u|(p)\right]s\,\,+\,\,o(s)
\\
&=\,W_R(p)\,-\,2\,\bigg[ \left(\frac{\Psi^2(p)+R^2}{\Psi^2(p)}\right) \sqrt{W(p)} \bigg] s\,+\,o(s)\,.
\end{align*}
To compare the two expansions, we recall that $W(p) = W_R(p)$ and $W\leq W_R$ in $N$, by Theorem~\ref{thm:min_pr_SD}. It follows that 
$$
2\,-\,\kappa(p)\,\sqrt{W_R(p)}\,\geq\,\frac{\Psi^2(p)+R^2}{\Psi^2(p)}\,,
$$
which can be rewritten as
$$
\kappa(p)\,\frac{|\Psi^2(p)-R^2|}{\Psi(p)}\,\leq\,\frac{\Psi^2(p)-R^2}{\Psi^2(p)}\,.
$$
Now, according to definition~\eqref{eq:pr_function_+}, if $\overline\tau(N) < \sqrt{2}$, then $\Psi(p)=\Psi_+(p)=\psi_+(0)=1$ and thus $\kappa(p) \leq 1$. On the other hand, according to definition~\eqref{eq:pr_function_-} if $\overline\tau(N) > \sqrt{2}$, then $\Psi(p)=\Psi_-(p)=\psi_-(0)=r_i(R)$, so that $\kappa(p) \leq - 1/r_i(R)$. This completes the proof of the proposition.
\end{proof}

We now pass to the curvature bounds that are taking place at the top stratum of ${\rm MAX}(u) \cap \overline{N}$. 
Notice that, while the previous results (Proposition~\ref{pro:curv_bound}, as well as Theorem~\ref{thm:min_pr_SD}) required $R(N)\neq 1$ in order to rule out the critical case $\overline{\tau}(N)=\sqrt{2}$, in the next result we will be able to address that special case as well with an argument based on a limiting procedure in the final part of the proof.

\begin{proposition}
\label{pro:curv_bound_Sigma}
Let $(\Om,u)$ be a solution to problem~\eqref{eq:prob_SD}, let $N$ be a connected component of $\Omega\setminus{\rm MAX}(u)$, and let $R = R(N) \in [0,1]$ be the expected core radius associated with the region $N$. Also assume that  Normalization~\ref{norm} is in force and denote by $\Sigma$ the $1$-dimensional top stratum of ${\rm MAX}(u) \cap \overline{N}$. Then, at any point $p \in \Sigma$, it holds
\begin{equation}
\label{eq:curv_bound_Sigma}
\kappa(p) \leq -\frac{1}{R}\,, \quad \hbox{if $\overline\tau(N)<\sqrt{2}$} \qquad \hbox{and} \qquad \kappa(p) \leq \frac{1}{R}\,, \quad \hbox{if $\overline\tau(N) \geq \sqrt{2}$} \, ,
\end{equation}
where $\kappa (p)$ is the curvature of $\Sigma$ at $p$, computed with respect to the exterior~\footnote{Exterior to $N$.} unit normal.
\end{proposition}

\begin{proof}
Lemma~\ref{le:WWm_expansions} provide us with the following Taylor expansions:
\begin{align*}
W\,&=\,4\,r^2\big[1\,+\, \kappa(p)\,r\big]\,+\,\mathcal{O}(r^4)\,,
\\
W_R\,&=\,4\,r^2\,\left[1\,+\,\left(\frac{\kappa(p)}{3}\,-\,\frac{2}{3 R}\right)r\right]+\mathcal{O}(r^4)\,, \quad \hbox{if $\overline\tau(N) <\sqrt{2}$}\,,
\\
W_R\,&=\,4\,r^2\,\left[1\,+\,\left(\frac{\kappa(p)}{3}\,+\,\frac{2}{3 R}\right)r\right]+\mathcal{O}(r^4)\,, \quad \hbox{if $\overline\tau(N) >\sqrt{2}$} \, ,
\end{align*}
where $\kappa (p)$ is the curvature of $\Sigma$ at $p$ computed with respect to the exterior unit normal, and $r(x)={\rm dist}(x,\Sigma)$ denotes the distance from $\Sigma$~\footnote{We agree that $r(x)>0$ for every $x \in N$.}.
Combining the above expansions with the gradient estimate $W \leq W_R$ obtained in Theorem~\ref{thm:min_pr_SD}, it is immediate to deduce that
$$
\kappa(p)\,\leq\,
\begin{dcases}
\frac{\kappa(p)}{3}\,-\,\frac{2}{3 R} & \hbox{ if $\overline\tau(N)<\sqrt{2}$\,,}
\\
\frac{\kappa(p)}{3}\,+\,\frac{2}{3 R}   & \hbox{ if $\overline\tau(N)>\sqrt{2}$\,.}
\end{dcases}
$$
This concludes the proof for $\overline\tau(N)\neq \sqrt{2}$.
The case $\overline\tau(N)=\sqrt{2}$ can be obtained by a limiting procedure: one just treats $N$ as if it were $\overline\tau(N)>\sqrt{2}$, defining the pseudo-radial function as in~\eqref{eq:pr_function_-}, with respect to an expected core radius $R_\ep=1-\ep$. By construction we have $W<W_{R_\ep}$ on $\pa\Omega\cap\overline{N}$.
Retracing then the proof of Theorem~\ref{thm:min_pr_SD}, one can easily check that the gradient estimate $W\leq W_{R_\ep}$ is still in force in the whole $N$. Hence, proceeding as above, we obtain the inequality
$$
\kappa(p)\,\leq\, \frac{1}{{R_\ep}}\,.
$$
Letting $\ep \to 0$, we deduce the desired bound.
\end{proof}

The following theorem can be seen as the {\em prelude} to the proof of Theorem~\ref{thm:B}, and it shows as, in a ring-shaped domain, the notions of expected core radii can be fruitfully employed to deduce a sharp and rigid pinching estimate on the curvature of the top stratum of ${\rm MAX}(u)$.

\begin{theorem}
\label{thm:curv_bound_Sigma}
Let $\Om$ being a ring-shaped domain, let $(\Om, u)$ be a solution to problem~\eqref{eq:prob_SD}, and according to~\eqref{eq:gammas}, let $\pa \Om \, = \, \Gamma_i \sqcup \Gamma_o$,
where $\Gamma_i$ and $\Gamma_o$ denote the inner and the outer connected component of the boundary of $\Om$, respectively. Assume that there exists a simple closed curve $\Sigma \subseteq{\rm MAX}(u)$ separating $\Omega$ into two regions $\Om_i$ and $\Om_o$, with $\pa \Om \cap \overline{\Om}_i = \Gamma_i$ and $\pa \Om \cap \overline{\Om}_o = \Gamma_o$. Also assume that $\overline\tau(\Om_o) <2$. Then, at any point $p \in \Sigma$, it holds
\begin{equation}
\label{eq:curv_pinch}
\frac{\sqrt{(u_{R_o})_{\max}}}{R_o} \, \leq \, \kappa(p) \sqrt{u_{\max}} \, \leq \, \frac{\sqrt{(u_{R_i})_{\max}}}{R_i} \, ,
\end{equation}
where $\kappa(p)$ is the curvature of $\Sigma$ computed with respect to the unit normal pointing outside $\Om_i$ (equiv. inside $\Om_o$), and, according to Definition~\ref{def:virtual_mass}, $R_i=R(\Gamma_i)$ and $R_o=R(\Gamma_o)$ are the expected core radii 
of $\Gamma_i$ and $\Gamma_o$, respectively. In particular, we have that $R_o\geq R_i$,
and the equality holds if and only if $(\Om,u)$ is equivalent to the ring-shaped model solution whose core radius is given by the common value of the two expected core radii. 
\end{theorem}

\begin{figure}
\centering
\includegraphics[scale=1]{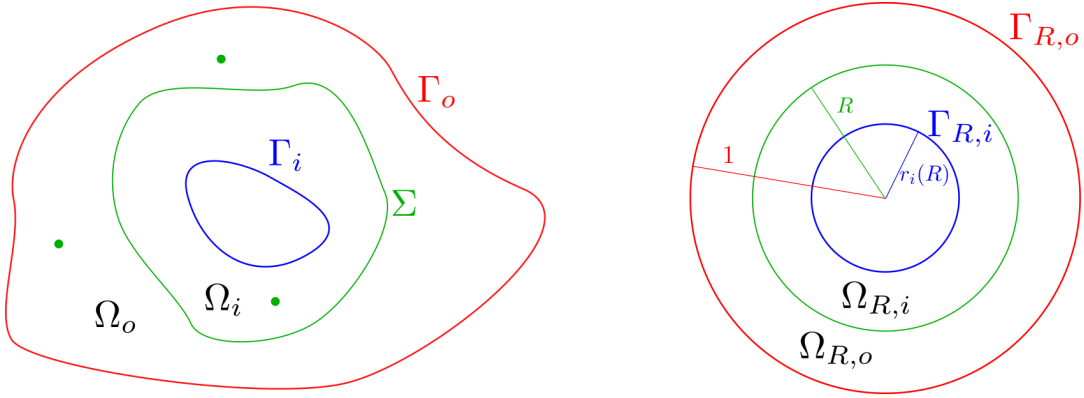}
\caption{
On the left is a generic example of a ring shaped domain satisfying the hypotheses of Theorem~\ref{thm:curv_bound_Sigma}. The thesis of  Theorem~\ref{thm:curv_bound_Sigma} is that under the additional assumption $R_o=R_i$, then necessarily the domain is rotationally symmetric, shaped as the picture on the right.
}
\label{fig:domain}
\end{figure}

\begin{proof}
Let us start from the analysis of the inner region $\Om_i$. As in the statement, let $R_i=R(\Gamma_i) = R (\Om_i)>0$ be the expected core radius of this region. Up to consider an equivalent pair $(\Om_{\lambda_i}, u_{\lambda_i})$ as in~\eqref{eq:scaling}, with $\lambda_i = \sqrt{(u_{R_i})_{\max} / u_{\max}}$, we may assume that Normalization~\ref{norm} is in force on the region $\Om_i$. Hence, applying  Proposition~\ref{pro:curv_bound_Sigma} in $\Om_i$ we obtain the upper bound
\begin{equation}
\kappa_i \,\leq\,
\begin{dcases}
-1/R_i & \hbox{if $\overline\tau(\Gamma_i)<\sqrt{2}$\,,}
\\
1/R_i & \hbox{if $\overline\tau(\Gamma_i)\geq \sqrt{2}$\,.}
\end{dcases}
\end{equation}
where $\kappa_i$ is the curvature of $\lambda_i \Sigma$, computed with respect to the unit normal pointing outside $\Om_i$. It is immediate to realize that only the second case is allowed, for if 
$\overline\tau(\Gamma_i)<\sqrt{2}$, then the curvature of the simple closed curve $\lambda_i \Sigma$ would be negative at each point. On the other hand, our choice of the unit normal implies that $\lambda_i \Sigma$ is oriented in the counterclockwise direction, so that the integral of $\kappa_i$ along $\lambda_i \Sigma$ must be equal to $2\pi>0$. Therefore $\overline\tau(\Gamma_i)\geq \sqrt{2}$, and in terms of the unnormalized quantities, the valid upper bound reads 
\begin{equation}
\label{eq:curv_bound_i}
\kappa(p) \sqrt{u_{\max}} \, \leq \, \frac{\sqrt{(u_{R_i})_{\max}}}{R_i} \, ,
\end{equation}
with $\kappa_i$ as in the statement if the theorem. A similar argument, based on Proposition~\ref{pro:curv_bound_Sigma}, leads to the desired lower bound
\begin{equation}
\label{eq:curv_bound_o}
\frac{\sqrt{(u_{R_o})_{\max}}}{R_o} \, \leq \, \kappa(p) \sqrt{u_{\max}} \, .
\end{equation}
The only relevant difference is that, when working in the outer region $\Om_o$, one cannot exclude a priori the case $\overline\tau(\Om_o)\geq \sqrt{2}$. This motivates the assumption $\overline\tau(\Om_o) < \sqrt{2}$ in the statement of the theorem. Combining~\eqref{eq:curv_bound_i} and~\eqref{eq:curv_bound_o}, we obtain~\eqref{eq:curv_pinch} and the fact that $R_o\geq R_i$ follows immediately from the fact that $R \mapsto {\sqrt{(u_{R})_{\max}}}/{R}$ is nonincreasing.

If $R_i=R=R_o$, for some $0<R<1$, then the curvature of $\Sigma$ is necessarily constant, and up to normalize everything so that $\sqrt{u_{\max}} = \sqrt{(u_{R})_{\max}}$, we have that $\kappa \equiv 1/R$ on the whole $\Sigma$. It follows that $\Sigma$ is a round circle of radius $R$. 
Now we observe that, in a neighborhood $U$ of $\Sigma$, our function $u$ solves the following initial value problem
\begin{equation}
\label{eq:prob_on_Sigma}
\begin{dcases}
\Delta u=-2, & \mbox{in }\, U\,,\\
\quad \!u= (u_R)_{\max}, & \mbox{on } \Sigma\,, \\
\frac{\pa u}{\pa\nu}=0, & \mbox{on } \Sigma\,.  
\end{dcases}
\end{equation}
The coefficients appearing in the above problem are clearly analytic, and $\Sigma$ is a noncharacteristic curve for $\De u=-2$, as there are no characteristic hypersurfaces for an elliptic PDE. Therefore, we can invoke the uniqueness statement in the Cauchy-Kovalevskaja Theorem, applied to the initial value problem~\eqref{eq:prob_on_Sigma}. On the other hand, since $\Sigma$ is a circle of radius $R$, we immediately check that the ring-shaped model solution solution~\eqref{eq:ST} of core radius $R$ also satisfies~\eqref{eq:prob_on_Sigma}. The uniqueness of the solution then implies that $(\Omega,u)$ must coincide with the model solution~\eqref{eq:ST} in a neighborhood of $\Sigma$. From the analyticity of $u$ it follows that the two solutions coincide everywhere.
\end{proof}

We are now ready to prove our main comparison result, namely Theorem~\ref{thm:B}.

\begin{proof}[Proof of Theorem~\ref{thm:B}]
We only need to argue that if $\Omega$ is a ring-shaped domain and $u$ has infinitely many maximum points, then a simple closed curve $\Sigma \subseteq {\rm MAX}(u)$ as in the statement of Theorem~\ref{thm:curv_bound_Sigma} does actually exist. To see this, we first recall from the {\L}ojasiewicz Structure Theorem~\cite{Lojasiewicz_2} (see also~\cite[Theorem~6.3.3]{Kra_Par}), that locally the set ${\rm Crit}(u)$ of the critical points of a nonconstant  real analytic function $u:\mathbb{R}^n \rightarrow \mathbb{R}$, has the structure of a real analytic sub-variety, whose strata may in principle be of any integer dimension between $0$ (points) and $(n-1)$ (top stratum). Moreover, the zero dimensional stratum is discrete and all the lower dimensional strata lie in the topological closure of the top stratum, whenever the latter is nonempty. In particular, in the case under consideration, we have that the zero dimensional stratum of ${\rm MAX}(u) \subset {\rm Crit}(u)$ must be finite, so that the ($1$-dimensional) top stratum is necessarily nonempty, as $u$ has infinitely many maximum points. It turns out that the top stratum of ${\rm MAX}(u)$ enjoys further regularity properties in the present context. In fact, as proven in~\cite[Corollary 3.4]{Bor_Chr_Maz}, if $\Sigma$ is a connected component of the top stratum and $\na^2 u$ is nowhere vanishing in $\overline\Sigma$, then $\Sigma = \overline\Sigma$, and $\Sigma$ is a real analytic simple closed curve. In particular ${\rm MAX}(u)$ is given by a finite number of isolated points and a finite number of isolated simple closed curves. An elementary application of the Strong Maximum Principle shows that the only possibility is that the top stratum of ${\rm MAX}(u)$ is given by one single simple closed curve $\Sigma$ dividing the ring-shaped domain $\Om$ into two region $\Om_i$ and $\Om_o$ as the ones described in the statement of Theorem~\ref{thm:curv_bound_Sigma}. This latter can now be invoked to complete the proof of Theorem~\ref{thm:B}.
\end{proof}

\section{Proof of Theorem~\ref{thm:const_grad}: Pohozaev identity and a priori length bounds}
\label{sec:theo_A}

In this section, we aim at proving our main classification result for ring-shaped solution to problem~\eqref{eq:prob_SD}, namely Theorem~\ref{thm:const_grad}. It states that if $u$ has infinitely many maximum points, then it must be rotationally symmetric, provided $|\na u|$ is constant on either $\Gamma_i$ or $\Gamma_o$. In order to achieve this result, we are going to exploit the latter condition mostly in two ways. The first one is presented in Subsection~\ref{sub:poho}, where by means of a Pohozaev identity we establish a natural correspondence between the geographical location of inner and outer regions and the possible range of their NWSS (see Proposition~\ref{pro:poho_ST}). The second one is presented in Subsection~\ref{sub:length}, where, combining the constancy of $|\na u|$ with the curvature bounds of the previous section, we obtain {\em a priori} bounds on the length of the boundary components of a given region. Theorem~\ref{thm:const_grad} will then be deduced in Subsection~\ref{sub:thmB} with the help of two other {\em a priori} bounds involving the length of the top stratum of ${\rm MAX}(u)$. These are proven in Propositions~\ref{pro:area_bound_Sigma} and~\ref{pro:area_lower_bound_mon} below.


\subsection{A Pohozaev identity}
\label{sub:poho}


In~\cite{Weinberger} an alternative proof of the result of Serrin~\cite{Serrin} is given using a suitable {\em Pohozaev identity}. In this subsection, we aim at providing a similar tool also in our context. 

Following~\cite{Gov_Ors}, a general version of the Pohozaev identity can be written on any $n$-dimensional Riemannian manifold $(M,g)$, for any vector field $X$ and any symmetric divergence free $(2,0)$-tensor field $B$:
\begin{equation}
\label{eq:pohozaev_ST}
\int_{M} X({\rm tr}(B))\, \rmd\mu\,=\,\frac{n}{2}\int_M\langle\mathring{B}|\mathcal{L}_X g\rangle \, \rmd\mu-n\int_{\pa M}\!\!\!\mathring{B}(X,\nu)\, \rmd\sigma\,,
\end{equation}
where $\nu$ is the outward unit normal to the boundary $\pa M$. We want to apply this identity to the manifold with boundary $\overline{N}$, where $N$ is a connected component of $\Om\setminus{\rm MAX}(u)$. 
To simplify the notation, we agree that 
\begin{equation*}
\partial \overline{N} \, = \, \Gamma_N \, \sqcup \, \Sigma_N \, , 
\end{equation*}
where $\Gamma_N = \partial \overline{N} \cap \pa \Om = \partial {N} \cap \pa \Om$ and $\Sigma_N = \partial \overline{N} \cap {\rm MAX}(u)$. Both $\Gamma_N$ and $\Sigma_N$ will be assumed to be smooth. The tensor $B$ is defined as follows
$$
B \, = \, du\otimes du+\left(2u-\frac{1}{2}|\D u|^2\right)g_{\R^2} \, .
$$
It is readily checked that if $u$ solves~\eqref{eq:prob_SD} then $B$ is divergence free. With this choices, 
the Pohozaev identity~\eqref{eq:pohozaev_ST} rewrites as
\begin{multline}
\label{eq:pohozaev_u_ST}
\int_{N}\left[ \langle \D_{\!\D u}X \, | \,\D u \rangle -\frac{1}{2}|\D u|^2{\rm div}(X)-2\langle\D u\,|\,X\rangle\right]\rmd\mu \, = 
\\
=\,-\int_{\pa \overline{N}}\left[\frac{1}{2}|\D u|^2\langle X\,|\,\nu\rangle-\langle\D u\,|\,X\rangle\langle\D u\,|\,\nu\rangle\right]\rmd\sigma.
\end{multline}
As $|\na u|= 0$ on $\Sigma_N$, the square bracket on the right hand side of~\eqref{eq:pohozaev_u_ST} is also vanishing on $\Sigma_N$. Observing that on $\Gamma_N$ the outer unit normal can be expressed as $\nu=-\D u/|\D u|$, the right hand side of~\eqref{eq:pohozaev_u_ST} can be rewritten as follows:
$$
-\int_{\pa \overline{N}}\left[\frac{1}{2}|\D u|^2\langle X\,|\,\nu\rangle-\langle\D u\,|\,X\rangle\langle\D u\,|\,\nu\rangle\right]\rmd\sigma\,=\,\frac{1}{2}\int_{\Gamma_N}|\D u|^2\langle X\,|\,\nu\rangle\rmd\sigma\,.
$$
Now we let $X$ be the position vector field of $\mathbb{R}^2$, i.e. $X(x) = x$ for every $x \in \mathbb{R}^2$. This choice implies that $\langle \D_{\!\D u}X \, | \,\D u \rangle = |\na u|^2$ and ${\rm div}(X)=2$, so that equation~\eqref{eq:pohozaev_u_ST} becomes
\begin{equation}
\label{eq:pohozaev_u_X_ST}
-4\int_{N}\langle\D u\,|\, x\rangle\,\rmd\mu\,=\,\int_{\Gamma_N}|\D u|^2\langle x\,|\,\nu\rangle\,\rmd\sigma \, .
\end{equation}
On the other hand, integrating by parts, we obtain the following identities
\begin{align}
\notag
\int_{N}\langle\D u\,|\, x\rangle\rmd\mu\,&=\,-\int_{N} u\,{\rm div}(x) \, \rmd\mu+\int_{\pa \overline{N}} u\,\langle x\,|\,\nu\rangle \, \rmd\sigma
\\
\label{eq:aux_poho_ST}
&=\,-2\int_{N} u\,\rmd\mu+\int_{\Gamma_N} \!\! u\,\langle x\,|\,\nu\rangle\rmd\sigma+\int_{\Sigma_N} \!\! u\,\langle x\,|\,\nu\rangle\rmd\sigma\,,
\\
\notag
2\,|N|\,=\,\int_{N}{\rm div}(x)\,\rmd\mu\,
&=\,\int_{\pa \overline{N}}\langle x\,|\,\nu\,\rangle \, \rmd\sigma
\\
\label{eq:aux2_poho_ST}
&=\int_{\Gamma_N} \!\! \langle x\,|\,\nu\rangle \, \rmd\sigma+\int_{\Sigma_N}\!\!\langle x\,|\,\nu\rangle \, \rmd\sigma\,.
\end{align}
Now, using the fact that $u=0$ on $\Gamma_N$ and $u=\umax$ on $\Sigma_N$, from~\eqref{eq:aux_poho_ST} we get
\begin{align*}
\int_{N}\langle\D u\,|\,x\rangle \, \rmd\mu\,&=\,-2\int_N u\,\rmd\mu+\umax\int_{\Sigma_N} \!\!\langle x\,|\,\nu\rangle \, \rmd\sigma
\\
&=\,-2\int_N u\,\rmd\mu+\umax\left[2\,|N|-\int_{\Gamma_N} \!\! \langle x\,|\,\nu\rangle \, \rmd\sigma\right]\,,
\end{align*}
where in the latter equality we have used~\eqref{eq:aux2_poho_ST}.
Substituting in~\eqref{eq:pohozaev_u_X_ST} and summarizing the above computations, we get
the following lemma
\begin{lemma}
Let $(\Om,u)$ be a solution to problem~\eqref{eq:prob_SD}, and let $N$ be a connected component of $\Omega\setminus{\rm MAX}(u)$, then it holds
\begin{equation}
\label{eq:pohozaev_u_X_final_ST}
8\int_{N}(u-\umax)\, \rmd\mu\,=\,\int_{\Gamma_N} \!\!\big(|\D u|^2-4\umax\big)\langle x\,|\,\nu\rangle \, \rmd\sigma\,,
\end{equation}
where $x$ is the position vector in $\mathbb{R}^2$ and $\Gamma_N= \pa {N} \cap \pa \Om$.
\end{lemma}

Building on formula~\eqref{eq:pohozaev_u_X_final_ST},  we obtain the following result, granting us that whenever $|\D u|$ is constant at the boundary, outer regions have NWSS below $2$ and inner regions have NWSS above the same threshold, in accordance to what happen for ring-shaped model solution.

\begin{proposition}
\label{pro:poho_ST}
Let $(\Om,u)$ be a solution to problem~\eqref{eq:prob_SD}, and let $N$ be  a connected component of $\Om\setminus{\rm MAX}(u)$. 
Suppose that $\Gamma_N = \partial N \cap \pa\Om$ is connected and $|\D u|$ is constant on $\Gamma_N$. Then $\overline\tau(\Gamma_N)\neq \sqrt{2}$ and:
\begin{itemize}
\item $ \overline\tau(\Gamma_N)<\sqrt{2}$ if and only if $N$ is outer, i.e. $\Gamma_N$ is oriented in the counterclockwise direction; 
\item $\overline\tau(\Gamma_N)>\sqrt{2}$ if and only if $N$ is inner, i.e. $\Gamma_N$ is oriented in the clockwise direction.
\end{itemize}
\end{proposition}

\begin{proof}
Since $|\D u|$ is constant on $N$, inequality~\eqref{eq:pohozaev_u_X_final_ST} gives
$$
0\,>\,8\int_{N}(u-\umax) \, \rmd\mu\,=\,\big(|\D u|^2-4\umax\big)\int_{\Gamma_N} \!\! \langle x\,|\,\nu\rangle \, \rmd\sigma\,=\,2\,\pi\,\eta\,\big(|\D u|^2-4\umax\big)\,,
$$
where $\eta=1$ if $\Gamma_N$ is oriented counterclockwise (i.e., if $N$ is outer) and $\eta=-1$ if $\Gamma_N$ is oriented clockwise (i.e., if $N$ is inner).
Therefore $\eta$ and $(|\D u|^2-4\umax)$ must have opposite sign. The wished result follows at once.
\end{proof}

In particular, notice that we have proven that, under the assumption that $|\D u|$ is constant on $\Gamma_N$, the critical case $\overline\tau(N)=\sqrt{2}$ is ruled out.


\subsection{Length bounds}
\label{sub:length}


In Proposition~\ref{pro:curv_bound}, we have proven a curvature bound holding at the points of the boundary $\Gamma_N = \pa N 	\cap \pa \Om$ where $|\D u|$ attains its maximum value. If we now assume that $|\D u|$ is constant on the boundary, then these curvature bounds hold at every boundary point. This fact leads to the following {\em a priori} estimates for the length of $\Gamma_N$.

\begin{proposition}
\label{pro:area_bound_Gamma}
Let $(\Om,u)$ be a solution to problem~\eqref{eq:prob_SD}, let $N$ be a connected component of $\Omega\setminus{\rm MAX}(u)$,  and let $R = R(N) \in [0,1)$ be the expected core radius associated with the region $N$. Also assume that  Normalization~\ref{norm} is in force. 
If $\Gamma_N = \pa N \cap \pa \Om$ is connected and $|\D u|$ is constant on $\Gamma_N$, then the following implications hold true:
\begin{itemize}
\item If $N$ is outer, then $ 2\pi \leq |\Gamma_N|$;

\smallskip

\item If $N$ is inner, then $|\Gamma_N|\,\leq\,2\pi r_i(R)$,
\end{itemize}
where $|\Gamma_N|$ denotes the length of $\Gamma_N$ and $r_i(R)>0$ is the inner radius defined in~\eqref{eq:ST}.                                                                                                                                   
\end{proposition}

\begin{proof}
In the present settings, it is easy to realize that the curvature bound~\eqref{eq:curv_bound} holds pointwise, so that we can integrate such an inequality on the whole $\Gamma_N$. In doing this, it is important to choose the orientation of $\Gamma_N$ according to the unit normal that has been employed to compute the curvature $\kappa$, namely the one that points outside $N$. Summoning Proposition~\ref{pro:poho_ST}, we have that if $N$ is outer, then $\overline\tau(N)=\overline\tau(\Gamma_N) <\sqrt{2}$. In particular, $\Gamma_N$ must be oriented in the counterlockwise direction, and first inequality in~\eqref{eq:curv_bound} is in force along any point of $\Gamma$. Integrating along $\Gamma$ leads to 
$$
2 \pi \leq |\Gamma_N|.
$$
Following the same path, one can easily prove also the second implication. The details are left to the interested reader.
\end{proof}

In order to prove Theorem~\ref{thm:const_grad}, we will also need an {\em a priori} bound for the length of $\Sigma_N = \partial \overline{N} \cap {\rm MAX}(u)$, in the case where $\Sigma_N$ is smooth. In contrast with Proposition~\ref{pro:area_bound_Gamma}, we notice that no assumption is made here about the constancy of $|\na u|$ at $\Gamma_N$. On the other hand the entire strategy of the proof is in a way very similar to the one of Proposition~\ref{pro:area_bound_Gamma} above.

\begin{proposition}
\label{pro:area_bound_Sigma}

Let $(\Om,u)$ be a solution to problem~\eqref{eq:prob_SD}, let $N$ be a connected component of $\Omega\setminus{\rm MAX}(u)$,  and let $R = R(N) \in [0,1)$ be the expected core radius associated with the region $N$. Also assume that  Normalization~\ref{norm} is in force. 
If $\Sigma_N = \pa \overline{N} \cap {\rm MAX}(u) \neq \emptyset$ is a smooth simple closed curve, then the following implications hold true:

\begin{itemize}
\item If $\overline\tau(N)<\sqrt{2}$, then $N$ is outer and $|\Sigma_N|\leq 2\pi R\,$;

\smallskip

\item If $\overline\tau(N)\geq \sqrt{2}$ and $N$ is inner, then $2\pi R \leq |\Sigma_N|\,$,
\end{itemize}
where $|\Sigma_N|$ denotes the length of $\Sigma_N$.  
\end{proposition}

\begin{proof}
We proceed as in the proof of Proposition~\ref{pro:area_bound_Gamma}. Integrating the curvature bounds given by Proposition~\ref{pro:curv_bound_Sigma} on $\Sigma_N$, and remembering that $\Sigma_N$ has been oriented with respect to the orientation induced by $N$, we get
$$
2\,\pi\,\eta\,\leq\,
\begin{dcases}
-|\Sigma_N|/R       & \hbox{ if $\overline\tau(N)<\sqrt{2}$\,,}
\\
|\Sigma_N|/R & \hbox{ if $\overline\tau(N)\geq \sqrt{2}$\,,}
\end{dcases}
$$
where $\eta=1$ if $\Sigma_N$ is oriented in the counterclockwise direction and $\eta=-1$ if $\Sigma_N$ is oriented in the clockwise direction. 

Notice that, since we are not assuming that $|\D u|$ is constant on $\Gamma_N$, Proposition~\ref{pro:poho_ST} does not apply, so that the usual upper (resp. lower) bound on $\overline\tau(N)$ is not necessarily equivalent to being $N$ an outer (resp. inner) region. Nevertheless, if $\overline\tau(N)<\sqrt{2}$, we can still deduce that $N$ must be outer, otherwise $\Sigma$ would be oriented in the counterclockwise direction and one would have  $2\pi R \leq -|\Sigma_N|$, which is clearly impossible.
\end{proof}

We conclude this subsection with the third length bound, where the length $|\Gamma_N|$ of the boundary portion of a given region $N$ is compared to the length $|\Sigma_N|$ of a smooth portion of the top stratum of ${\rm MAX}(u)$.
\begin{proposition}
\label{pro:area_lower_bound_mon}

Let $(\Om,u)$ be a solution to problem~\eqref{eq:prob_SD}, let $N$ be a connected component of $\Omega\setminus{\rm MAX}(u)$,  and let $R = R(N) \in [0,1)$ be the expected core radius associated with the region $N$. Also assume that  Normalization~\ref{norm} is in force. 
If $\Sigma_N = \pa \overline{N} \cap {\rm MAX}(u) \neq \emptyset$ is a smooth simple closed curve, and $\Gamma_N = \pa N \cap \pa \Om$ is connected, then the following implications hold true:
\begin{itemize}
\item If $\overline\tau(N)<\sqrt{2}$, then $|\Sigma_N| /R \leq |\Gamma_N| \,$;

\smallskip

\item If $\overline\tau(N)> \sqrt{2}$, then $|\Sigma_N| /R \leq |\Gamma_N| /r_i(R)$,
\end{itemize}
where $|\Sigma_N|$ and $|\Gamma_N|$ denote the lengths of $\Sigma_N$ and $\Gamma_N$, respectively. Moreover, if the equality holds, 
then $(\Omega,u)$ corresponds to the ring-shaped model solution~\eqref{eq:ST} with core radius $R$.
%
%
\end{proposition}

\begin{proof} The first step of the proof consist in establishing a useful integral identity. To this purpose, let us set $N_\ep=N\cap\{0\leq u\leq \umax-\ep\}$, where the small enough parameter $\ep>0$ is chosen in such a way that the level set $u=\umax-\ep$ is regular. Sard's Theorem grants us that the regular values of $u$ are dense, and so we will then let $\ep$ go to zero along a sequence of regular values.
Using the Divergence Theorem and recalling that $\De u=-2$, we compute
\begin{align*}
\int_{N_\ep}\frac{-2}{\Psi^2-R^2}\,\rmd\mu\,=\,\int_{N_\ep} \frac{\De u}{\Psi^2-R^2}\,\rmd\mu\,&=\,-\int_{N_\ep}\left\langle\D\left(\frac{1}{\Psi^2-R^2}\right)\,\bigg|\,\D u\right\rangle\,\rmd\mu\,+\,\int_{\pa N_\ep}\frac{\langle \D u\,|\,\nu\rangle}{\Psi^2-R^2}\,\rmd\sigma
\\
&=\,-\int_{N_\ep}\frac{2\,\Psi^2}{(\Psi^2-R^2)^3}\,|\D u|^2\,\rmd\mu\,+\,\int_{\pa N_\ep}\frac{\langle \D u\,|\,\nu\rangle}{\Psi^2-R^2}\,\rmd\sigma
\end{align*}
where $\nu$ is the outward unit normal to $\pa N_\ep$ and $\Psi$ is the pseudo-radial function, defined as usual by either~\eqref{eq:pr_function_+} or~\eqref{eq:pr_function_-} depending on whether $\overline\tau(N)<\sqrt{2}$ or $\overline\tau(N)>\sqrt{2}$, respectively. Notice that $\pa N_\ep=\Gamma_N\cup\Sigma_{N,\ep}$, where $\Sigma_{N,\ep}=N\cap\{u=\umax-\ep\}$ is a compact hypersurface that gets closer to $\Sigma_N$ as $\ep$ goes to $0$. It is clear that $\nu=-\D u/|\D u|$ on $\Gamma_N$ and $\nu=\D u/|\D u|$ on $\Sigma_{N,\ep}$. Setting $W=|\D u|^2$ as usual, and recalling the expression~\eqref{eq:Wm} of $W_R$ in terms of $\Psi$, the identity above can be written as
\begin{equation}
\label{eq:crucial_mon}
\int_{N_\ep}\frac{2\,\Psi^2}{(\Psi^2-R^2)^3}\,(W-W_R)\,\rmd\mu\,=\,-\,\int_{\Gamma_N}\frac{|\D u|}{\Psi^2-R^2}\,\rmd\sigma\,+\,\int_{\Sigma_{N,\ep}}\frac{|\D u|}{\Psi^2-R^2}\,\rmd\sigma\,.
\end{equation}
Theorem~\ref{thm:min_pr_SD} tells us that the factor $(W-W_R)$ appearing in the right hand side of~\eqref{eq:crucial_mon} is nonpositive. Since $\Psi^2-R^2>0$ if $\overline\tau(\Gamma_N)<\sqrt{2}$ and $\Psi^2-R^2<0$ if $\overline\tau(\Gamma_N)>\sqrt{2}$, the left hand side of~\eqref{eq:crucial_mon} is nonpositive if $\overline\tau(\Gamma_N)<\sqrt{2}$ and it is nonnegative if $\overline\tau(\Gamma_N)>\sqrt{2}$. Furthermore, if the left hand side vanishes, then $W=W_R$ in the whole $N_\ep$, and the rigidity statement of Theorem~\ref{thm:min_pr_SD} forces $(\Omega,u)$ to be the ring-shaped model solution with core radius $R$. 

We now proceed with the analysis of the right hand side of~\eqref{eq:crucial_mon}. If $\overline\tau(\Gamma_N)<\sqrt{2}$, we have that $\Psi \equiv 1$ on $\Gamma_N$ and $\max_{\Gamma_N}|\D u|={\sqrt{W_R}}_{|_{\Gamma_N}}=1- R^2$, hence
$$
\int_{\Gamma_N} \frac{|\D u|}{\Psi^2-R^2}\,\rmd\sigma\,=\,\int_{\Gamma_N} \frac{|\D u|}{\max_{\Gamma_N}|\D u|}\,\rmd\sigma\,\leq\,|\Gamma_N|\,.
$$
If instead $\overline\tau(\Gamma_N)>\sqrt{2}$, we have that $\Psi=r_i(R)$ on $\Gamma_N$ and $\max_{\Gamma_N}|\D u|={\sqrt{W_R}}_{|_{\Gamma_N}}=(R^2-r_i^2)/r_i$, hence
$$
\int_{\Gamma_N} \left|\frac{\D u}{\Psi^2-R^2}\right|\,\rmd\sigma\,=\,\frac{1}{r_i}\int_{\Gamma_N} \frac{|\D u|}{\max_{\Gamma_N}|\D u|}\,\rmd\sigma\leq\frac{|\Gamma_N|}{r_i}
$$
Finally, we compute the limit of the integral on $\Sigma_{N,\ep}$, when $\ep \to 0$. As specified above, we understand that such a limit is taken along a sequence of regular values of $u$.
Recalling again the definitions of $W$ and $W_R$ and the fact that $\Psi^2-R^2$ is positive (respectively, negative) in $N$ if $\overline\tau(\Gamma_N)<\sqrt{2}$ (respectively, $\tau(\Gamma_N)>\sqrt{2}$), we can rewrite this limit as follows
$$
\lim_{\ep\to 0^+}\int_{\Sigma_{N,\ep}}\frac{|\D u|}{\Psi^2-R^2}\,\rmd\sigma\,=\,\lim_{\ep\to 0^+}\int_{\Sigma_{N,\ep}}\pm\frac{1}{\Psi}\sqrt{\frac{W}{W_R}}\,\rmd\sigma\,,\,=\,\pm\frac{1}{R}\lim_{\ep\to 0^+}\int_{\Sigma_\ep}\sqrt{\frac{W}{W_R}}\,\rmd\sigma\,,
$$
where the sign $\pm$ is $+$ if $\overline\tau(\Gamma_N)<\sqrt{2}$ and it is $-$ if $\overline\tau(\Gamma_N)>\sqrt{2}$.
We can then apply Lemma~\ref{le:W_Wm}, which tells us that the limit of this integral over $\Sigma_{N,\ep}$ is greater than or equal to the area of $\Sigma_N$.
The wished result is now easily obtained combining the latter observations with the integral identity~\eqref{eq:crucial_mon}.
\end{proof}


\subsection{Proof of Theorem~\ref{thm:const_grad}}
\label{sub:thmB}

We now have all the ingredients to prove Theorem~\ref{thm:const_grad}. We rewrite hereafter the statement, for the ease of the reader

\begin{theorem}
Let $(\Om,u)$ be a solution to problem~\eqref{eq:prob_SD} such that $\Om$ is a ring-shaped domain and $u$ has infinitely many maximum points. Assume that $|\nabla u|$ is constant on either the inner or the outer boundary component. Then, up to translations and rescaling, $(\Om, u)$ corresponds to a ring-shaped model solution~\eqref{eq:ST}.
%
\end{theorem}

\begin{proof}
First observe that, arguing as in the proof of Theorem~\ref{thm:B}, one gets that the $1$-dimensional top stratum of ${\rm MAX}(u)$ is necessarily nonempty and given by a smooth simple closed curve, that we denote by $\Sigma$.
As in the statement of Theorem~\ref{thm:curv_bound_Sigma}, 
$\Sigma \subseteq{\rm MAX}(u)$ is separating $\Omega$ into two regions: one inner region $\Om_i$ and one outer region $\Om_o$, with $\pa \Om \cap \overline{\Om}_i = \Gamma_i$ and $\pa \Om \cap \overline{\Om}_o = \Gamma_o$.

Let us first assume that $|\na u|$ is constant on the outer boundary component $\Gamma_o$. In this case, we have that $(\Om_o,u)$ obeys 
\begin{equation*}
\begin{dcases}
\Delta u=-2 & \quad\mbox{in } \Om_o \, ,\\
\,\,\,\,\, u=0 & \quad\mbox{on } \Gamma_o \, ,\\
\,\,\,\,\, u=\umax & \quad\mbox{on } \Sigma\, .\\
\end{dcases}
\end{equation*}
Moreover, $|\na u|$ is locally constant on $\pa \Om_o= \Gamma_o \sqcup \Sigma$, and in particular 
\begin{equation*}
\frac{\pa u}{\pa\nu} \, = \, 0 \, \qquad \hbox{on
$\Sigma$} \, , 
\end{equation*}
where $\nu$ is the outer unit normal to $\Sigma$. In other words, the assumptions of Theorem~\ref{thm:Syr} are sharply satisfied. It follows that, up to translations and rescalings, the solution $(\Om_o, u)$ corresponds to (a portion of) a ring-shaped model solution, and so does $(\Om, u)$, by analyticity.

Let us now assume that $|\na u|$ is constant on the inner boundary component $\Gamma_i$. In this case, we necessarily have that $\overline\tau(\Gamma_i)>\sqrt{2}$, in virtue of Proposition~\ref{pro:poho_ST}. We can then employ the length bounds of the previous subsection, and more precisely Propositions~\ref{pro:area_bound_Gamma},~\ref{pro:area_bound_Sigma} and~\ref{pro:area_lower_bound_mon} with $N= \Om_i$ and $\Gamma_N = \Gamma_i$, to obtain the following chain of inequalities:
\begin{equation*}
2\pi \, \leq \, \frac{|\Sigma|}{R} \, \leq \, \frac{|\Gamma_i|}{r_i} \, \leq \, 2\pi \,,
\end{equation*}
where $R$ is the expected core radius of $\Om_i$, and $r_i=r_i(R)$.
It follows that all the above inequalities are saturated. In particular it holds that 
$$
\frac{|\Sigma|}{R} \, = \, \frac{|\Gamma_i|}{r_i} \, , 
$$
so that we can invoke the rigidity statement of Proposition~\ref{pro:area_lower_bound_mon} to conclude.
\end{proof}

\section{Proof of Theorem~\ref{thm:bifurcation}}
\label{Sec:bif}

This section is dedicated to the proof of Theorem~\ref{thm:bifurcation}, that, we recall, tells us that there are infinite solutions $(\Omega,u)$ to problem~\eqref{eq:prob_SD} that are not rotationally symmetric and such that $\Omega$ is a ring-shaped domain and $|\D u|$ is locally constant on $\pa \Omega$. 
In fact, we will be able to prove a far more precise statement, namely Theorem~\ref{thm:A_general}, that will give us a better picture of such exotic solutions. 
The proof is a modification of the arguments given in~\cite{Kam_Sci}, where the same result is proven for a similar problem. It will however be quite clear from our proof that there are several technical complications that make the computations much more delicate in our case. In an attempt to keep the presentation as clear-cut as possible, we will try to refer to~\cite{Kam_Sci} whenever possible, stressing only the main differences. To this aim, in order to have a notation as similar as possible to the one in~\cite{Kam_Sci}, we will use the notation $\lambda=r_i(R)\in(0,1)$ and we will consider $R$ as a function of $\lambda$. Notice in fact that $\lambda$, $R$ are related by
\begin{equation}
\label{eq:rel_lambda_m}
1-\lambda^2+2R^2\log\lambda \,=\,0\,,\quad \hbox{ or equivalently }\quad R^2\,=\,\frac{1-\lambda^2}{-2\,\log\lambda}\,.
\end{equation}
For any $\lambda\in(0,1)$, we denote by
$$
\Om_\lambda \,=\, \{\lambda\,<\,|x|\,<\,1\}\,,\quad u_\lambda\,=\,\frac{1-|x|^2}{2}\,+\,R^2\,\log |x|
$$
the rotationally symmetric solution with core radius $R$, by $\Gamma_{\lambda}=\{|x|=\lambda\}$, $\Gamma_{1}=\{|x|=1\}$ the two connected components of $\pa\Omega$, and by $c_\lambda^i=R^2/\lambda - \lambda$, $c_\lambda^o=1-R^2$ the constant value of $|\D u_\lambda|$ on $\Gamma_{\lambda}$ and $\Gamma_1$, respectively.

Let now $\alpha\in(0,1)$ and ${\bf v}=(v_1,v_2)\in\left(\mathscr{C}^{2,\alpha}(\Sph^1)\right)^2$, consider the domain
$$
\Om_\lambda^{\bf v} \,=\, \left\{\lambda\,+\,v_1(x/|x|)\,<\,|x|\,<\,1 \,-\, v_2(x/|x|) \right\}
$$
and let $u_\lambda^{\bf v}:\Om_\lambda^{\bf v}\to\R$ be the solution to the problem

\begin{equation}
\label{eq:prob_v}
\begin{dcases}
\Delta u=-2, & \mbox{in } \Om_\lambda^{\bf v}\,,\\
u=0, & \mbox{on } \partial \Om_\lambda^{\bf v}\,.
\end{dcases}
\end{equation}

Let $\Gamma_{\lambda}^{\bf v}$ and $\Gamma_{1}^{\bf v}$ be the interior and exterior connected component of $\pa\Om_\lambda^{\bf v}$, respectively.
We can now state the main result of this section.

\begin{theorem}
\label{thm:A_general}
Let $\alpha\in (0,1)$. There is a strictly increasing sequence $\{\lambda_\sigma\}_{\sigma=1}^\infty$ of positive real numbers with $\lim_{\sigma\to+\infty}\lambda_\sigma=1$, such that, for every $\sigma\in\N$, there exists $\ep>0$ and a smooth curve
\begin{align*}
(-\ep,\ep)&\longrightarrow\left(\mathscr{C}^{2,\alpha}(\Sph^1)\right)^2\times (0,1)
\\
s\ \ \ &\longmapsto\ \ \ \ \ \ \ \ \ ({\bf v}(s)\,,\,\lambda(s))
\end{align*}
with ${\bf v}(0)=(0,0)$, $\lambda(0)=\lambda_{\sigma}$, such that, for every $s\in(-\ep,\ep)$, in the notations introduced above, the solution $u_{\lambda(s)}^{{\bf v}(s)}$ to~\eqref{eq:prob_v} in $\Omega_{\lambda(s)}^{{\bf v}(s)}$ satisfies
\begin{equation*}
\,\big|\D u_{\lambda(s)}^{{\bf v}(s)}\big|=c_{\lambda(s)}^i \quad\mbox{on }\Gamma_{\lambda(s)}^{{\bf v}(s)},\qquad\quad
\,\big|\D u_{\lambda(s)}^{{\bf v}(s)}\big|=c_{\lambda(s)}^o \quad\mbox{on }\Gamma_1^{{\bf v}(s)}\,,
\end{equation*}
and $\Omega_{\lambda(s)}^{{\bf v}(s)}$ is not rotationally symmetric for any $s\neq 0$.
\end{theorem}

It is clear that this result implies Theorem~\ref{thm:bifurcation} at once.
The rest of the section is therefore dedicated to the proof of Theorem~\ref{thm:A_general}.
Let $U\subseteq \left(\mathscr{C}^{2,\alpha}(\Sph^1)\right)^2$ be a small neighborhood of ${\bf 0} = (0,0)$, and define the function
\begin{equation}
\label{eq:F}
\begin{aligned}
F_\lambda:\, U &\longrightarrow \left(\mathscr{C}^{1,\alpha}(\Sph^1)\right)^2
\\
{\bf v} &\longmapsto
\left(\frac{\pa u_\lambda^{\bf v}}{\pa\nu}_{|_{\Gamma_{\lambda}^{\bf v}}}\,-\,c_\lambda^i\ ,\ \frac{\pa u_\lambda^{\bf v}}{\pa\nu}_{|_{\Gamma_{1}^{\bf v}}}\,-\,c_\lambda^o \right) \,,
\end{aligned}
\end{equation}
where $\nu$ is the inward unit normal to $\pa\Omega_\lambda^{\bf v}$.

Our aim is that of {\em finding ${\bf v}$ that is not trivial (meaning that ${\bf v}\neq {\bf 0}$ and it is not a simple translation) and such that $F_\lambda({\bf v})={\bf 0}$}.
We start by linearizing $F_\lambda$, that is, we consider the function
$$
L_\lambda({\bf w})\,=\,\lim_{t\to 0}\frac{F_\lambda(t\,{\bf w})}{t}\,.
$$
Proceeding in the same way as in~\cite[Proposition~3.1]{Kam_Sci}, we obtain the following expression for $L_\lambda$:
$$
L_\lambda({\bf w})\,=\,
\left(w_1\,\frac{\pa^2 u_\lambda}{\pa r^2}_{|_{\Gamma_{\lambda}}}\!\!\!-\,\frac{\pa \phi_\lambda^{\bf w}}{\pa\nu}_{|_{\Gamma_{\lambda}}} ,\ w_2\,\frac{\pa^2 u_\lambda}{\pa r^2}_{|_{\Gamma_1}}\!\!\!-\,\frac{\pa \phi_\lambda^{\bf w}}{\pa\nu}_{|_{\Gamma_1}} \right)\,,
$$
where $\phi_\lambda^{{\bf w}}$ is the solution to
$$
\begin{dcases}
\Delta \phi=0, & \mbox{in } \Om_\lambda\,,\\
\phi=c_\lambda^i \,w_1, & \mbox{on } \Gamma_\lambda\,,\\
\phi=c_\lambda^o\,w_2, & \mbox{on } \Gamma_1\,.
\end{dcases}
$$
In order to obtain a more explicit formula for $L_\lambda$, it is convenient to restrict the attention to spherical harmonics. For a given integer $k\in\N_0$, let then $Y\in\mathscr{C}^\infty(\Sph^1)$, be a nontrivial solution to 
$$
\De_{\Sph^1}Y \,+\, k^2\,Y = 0 \,,
$$ 
and let $W\subseteq \left(\mathscr{C}^{\infty}(\Sph^1)\right)^2$ be the subspace generated by $(Y,0)$ and $(0,Y)$. Finally, fixed $\lambda\in(0,1)$, let
$$
{\bf e}_1\,=\,\left( \frac{1}{\sqrt{\lambda}}\,Y\,,\,0\right)\,,\quad {\bf e}_2\,=\,\left(0\,,\, Y\right)
$$
be a base of $W$, orthonormal with respect to the scalar product
\begin{equation}
\label{eq:scal_prod} 
\langle {\bf w},{\bf z}\rangle_\lambda\,=\,\lambda\int_{\Sph^1} w_1 \,z_1\,d\sigma\,+\,\int_{\Sph^1} w_2\,z_2\,d\sigma\,.
\end{equation}
For any element ${\bf w}=a\, {\bf e}_1 + b\, {\bf e}_2\in W$, one can compute
$$
\phi_\lambda^{\bf w}=
\begin{dcases}
\left(a\,\frac{c_\lambda^i}{\sqrt{\lambda}}\,\frac{|x|^k - |x|^{-k}}{\lambda^k - \lambda^{-k}}\,+\,b\,c_\lambda^o\,\frac{\lambda^{k}|x|^{-k} - \lambda^{-k} |x|^{k}}{\lambda^k - \lambda^{-k}}\right)Y   &   \hbox{ if } k>0\,,
\\
\left(a\,\frac{c_\lambda^i}{\sqrt{\lambda}}\,\frac{\log|x|}{\log\lambda}\,+\,b\,c_\lambda^o\,\frac{\log\lambda-\log|x|}{\log\lambda}\right)Y   &   \hbox{ if } k=0\,,
\end{dcases}
$$
In particular, ${\rm Im}({L_\lambda}_{|_{W}})\subseteq W$ and the matrix associated to the restriction ${L_\lambda}_{|_{W}}$ with respect to the basis ${\bf e}_1$, ${\bf e}_2$ can be computed as
$$
M_{\lambda,k}\,=\,\widetilde{M}_{\lambda,k}
\,-\,2\,{\rm Id}\,,\quad \hbox{where } 
\widetilde{M}_{\lambda,k}\,=\,
\begin{dcases}
\begin{bmatrix}
\frac{R^2-\lambda^2}{\lambda^2}(k\coth\omega -1) & -\frac{k}{\sqrt{\lambda}}(1-R^2)\frac{1}{\sinh\omega} \\
-\frac{k}{\sqrt{\lambda}}\frac{R^2-\lambda^2}{\lambda}\frac{1}{\sinh\omega} & (1-R^2)(k\coth\omega +1)
\end{bmatrix}   &   \hbox{ if } k>0\,,
\\
\begin{bmatrix}
\frac{R^2-\lambda^2}{\lambda^2}(-\frac{1}{\log\lambda} -1) & \frac{1}{\sqrt{\lambda}}(1-R^2)\frac{1}{\log\lambda} \\
\frac{1}{\sqrt{\lambda}}\frac{R^2-\lambda^2}{\lambda}\frac{1}{\log\lambda} & (1-R^2)(-\frac{1}{\log\lambda} +1)
\end{bmatrix}    &   \hbox{ if } k=0\,.
\end{dcases}
$$
Here we have denoted by $\omega$ the function satisfying $e^\omega=\lambda^{-k}$. While we are of course interested only to integer values of $k$, the matrix $M_{\lambda,k}$ makes sense for any real value of $k\geq 0$. Notice that $M_{\lambda,k}$ is analytic in both variables $(\lambda,k)\in(0,1)\times(0,+\infty)$, and it can be checked easily that $\lim_{k\to 0^+} M_{\lambda,k}=M_{\lambda,0}$, which implies that $M_{\lambda,k}$ is continuous up to $(\lambda,k)\in(0,1)\times\{0\}$.

If we denote by 
\begin{equation}
\label{eq:TD}
\begin{aligned}
T_{\lambda,k}\,&=\,R^2\left(\frac{1}{\lambda^2}-1\right) \,k\, \coth\omega\,+\,2\,-\,R^2\,-\,\frac{R^2}{\lambda^2}\,.
\\
D_{\lambda,k}\,&=\,\left(\frac{R^2}{\lambda^2}-1\right)(1-R^2)(k^2-1)
\end{aligned}
\end{equation} 
the trace and determinant of $\widetilde{M}_{\lambda,k}$, following a computation that is completely analogous to the one leading to estimate~(4.17) in~\cite{Kam_Sci}, we obtain
\begin{equation}
\label{eq:delta_estimate}
T_{\lambda,k}^2-4D_{\lambda,k}\,=\,\left\{k\left(\frac{c_\lambda^i}{\lambda}-c_\lambda^o\right)\coth\omega\,+\,\left(\frac{c_\lambda^i}{\lambda}+c_\lambda^o\right)\right\}^2\,+\,4\,k^2\,\frac{c_\lambda^i c_\lambda^o}{\lambda}\frac{1}{\sinh^2\omega}>0\,.
\end{equation}

As a consequence, the eigenvalues of $M_{\lambda,k}$, given by
\begin{equation}
\label{eq:eigenv_expr}
\mu_1(\lambda,k)\,=\,\frac{T_{\lambda,k}-\sqrt{T_{\lambda,k}^2-4D_{\lambda,k}}}{2}\,-\,2\,,\quad \mu_2(\lambda,k)\,=\,\frac{T_{\lambda,k}+\sqrt{T_{\lambda,k}^2-4D_{\lambda,k}}}{2}\,-\,2\,,
\end{equation}
are distinct real numbers. Furthermore, $\mu_1(\lambda,k)$ and $\mu_2(\lambda,k)$ have the same regularity as $M_{\lambda,k}$, namely they are analytic for $(\lambda,k)\in(0,1)\times(0,+\infty)$ and continuous up to $(\lambda,k)\in(0,1)\times\{0\}$.

For $k=1$, we can easily compute them explicitly:
$$
\mu_1(\lambda,1)\,=\,-2\,,\qquad \mu_2(\lambda,1)\,=\,0\,.
$$
Another simple computation shows that for any $k> 1$ it holds
$$
\lim_{\lambda\to 1} \mu_1(\lambda,k)\,=\,-2\,,\qquad \lim_{\lambda\to 1}\mu_2(\lambda,k)\,=\,0\,.
$$
Concerning the limit when $\lambda\to 0$, we first observe that, from the relation~\eqref{eq:rel_lambda_m} between $R$ and $\lambda$, it easily follows that $R/\lambda$ diverges to $+\infty$ as $\lambda\to 0$. Recalling the explicit expressions~\eqref{eq:TD} of $T_{\lambda,k}$ and $D_{\lambda,k}$, for any $k>1$, we easily obtain the following behavior for $\lambda$ close to zero:
$$
T_{\lambda,k}\,=\,\frac{R^2}{\lambda^2}(k-1)\,+\,o(R^2/\lambda^2)\,,\qquad D_{\lambda,k}\,=\,\frac{R^2}{\lambda^2}(k^2-1)\,+\,o(R^2/\lambda^2)
$$
As a consequence, both $T_{\lambda,k}$ and $D_{\lambda,k}$ diverge to $+\infty$ as $\lambda\to 0$, so that in particular $\mu_2(\lambda,k)\to +\infty$ when $\lambda\to 0$. Concerning the first eigenvalue, with some easy computations we find:

\begin{align*}
\lim_{\lambda\to 0}\mu_1(\lambda,k)\,&=\,\frac{1}{2}\lim_{\lambda\to 0}\left\{T_{\lambda,k}\left[1-\sqrt{1-4\frac{D_{\lambda,k}}{T^2_{\lambda,k}}}\right]\right\}-2
\\
&=\,\frac{1}{2}\lim_{\lambda\to 0}\left\{T_{\lambda,k}\frac{4\frac{D_{\lambda,k}}{T^2_{\lambda,k}}}{1+\sqrt{1-4\frac{D_{\lambda,k}}{T^2_{\lambda,k}}}}\right\}-2
\\
&=\,\lim_{\lambda\to 0} \frac{D_{\lambda,k}}{T_{\lambda,k}}-2\,=\,\frac{k^2-1}{k-1}-2\,=\,k-1>0\,.
\end{align*}
In particular notice that the limits of $\mu_1(\lambda,k)$ as $\lambda\to 0$ and $\lambda\to 1$ have different sign, from which it follows that, for any positive $k\in\N$, {\em there is at least one value $\lambda_k$ such that $\mu_1(\lambda_k,k)=0$}. We will come back to this point later, in Proposition~\ref{pro:main_bifurcation}, where we will show that such points possess a number of crucial properties. Before stating that proposition and dealing with its proof, we need a couple of preparatory results, concerning the monotonicity of the eigenvalues with respect to $\lambda$ and $k$. 

\begin{lemma}
\label{le:monotonicity_k}
The eigenvalues $\mu_{1}(\lambda,k)$, $\mu_{2}(\lambda,k)$ are monotonically increasing in $k$.
\end{lemma} 

\begin{proof}
We want to follow the same strategy used in~\cite[Lemma~4.4]{Kam_Sci}. In order to do that, we need to work with a symmetric matrix. While it is true that $M_{\lambda,k}$ is not symmetric, we can easily find a symmetric matrix with the same eigenvalues, namely:
$$
M_{\lambda,k}^S\,=\,\widetilde{M}_{\lambda,k}^S
\,-\,2\,{\rm Id}\,,\quad \hbox{where } 
\widetilde{M}_{\lambda,k}^S\,=\,
\begin{bmatrix}
\frac{R^2-\lambda^2}{\lambda^2}(k\coth\omega -1) & -\frac{k}{\lambda}\frac{\sqrt{1-R^2}\sqrt{R^2-\lambda^2}}{\sinh\omega} \\
-\frac{k}{\lambda}\frac{\sqrt{1-R^2}\sqrt{R^2-\lambda^2}}{\sinh\omega} & (1-R^2)(k\coth\omega +1)
\end{bmatrix}\,.
$$
It is clear that $M_{\lambda,k}^S$ has the same trace and determinant, and thus the same eigenvalues, of $M_{\lambda,k}$. Arguing as in~\cite[Lemma~4.4]{Kam_Sci}, in order to prove that $\mu_1(\lambda,k)$ and $\mu_2(\lambda,k)$ are monotonically increasing in $k$, it is sufficient to show that the matrix
$$
\pa_k M_{\lambda,k}^S\,=\,\begin{bmatrix}
\frac{R^2-\lambda^2}{\lambda^2}\left(\coth\omega -\frac{\omega}{\sinh^2\omega}\right) & \frac{\sqrt{1-R^2}\sqrt{R^2-\lambda^2}}{\lambda}\left(\frac{\omega\cosh\omega}{\sinh^2\omega}-\frac{1}{\sinh\omega}\right) \\
\frac{\sqrt{1-R^2}\sqrt{R^2-\lambda^2}}{\lambda}\left(\frac{\omega\cosh\omega}{\sinh^2\omega}-\frac{1}{\sinh\omega}\right) & (1-R^2)\left(\coth\omega -\frac{\omega}{\sinh^2\omega}\right)
\end{bmatrix}
$$
is positive definite. This is done exactly as in~\cite{Kam_Sci}, so we avoid to give the details.
\end{proof}

\begin{figure}
\centering
\includegraphics[scale=0.5]{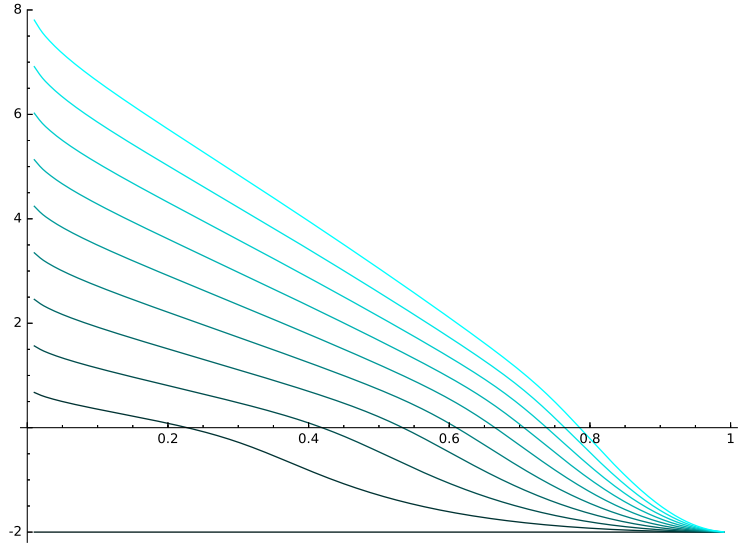}
\caption{Plot of the eigenvalue $\mu_1(\lambda,k)$ as a function of $\lambda$ for the integers $k$ ranging from $1$ (dark blue) to $10$ (light blue).}
\label{fig:mu1}
\end{figure}

In~\cite{Kam_Sci} it is also proven that the first eigenvalue $\mu_1(\lambda,k)$ is monotonically decreasing in $\lambda$. 
This seems to hold true in our framework as well (see Figure~\ref{fig:mu1}), but it appears to be harder to prove, as the explicit expression for $\pa\mu_1/\pa\lambda$ is more complicated. Luckily we do not need such a strong result, but it will be enough to show that the derivative $\pa\mu_1/\pa\lambda$ is negative at the points where $\mu_1$ vanishes. Namely, we will need the following:

\begin{lemma}
\label{le:monotonicity_lambda}
For any $k>1$, if $\lambda_k$ is such that $\mu_1(\lambda_k,k)=0$, then
$$
{\frac{\pa \mu_1}{\pa \lambda}}(\lambda_k,k)< 0\,.
$$
\end{lemma}

\begin{proof}
Denote by $D'_{\lambda,k}$, $T'_{\lambda,k}$ the derivatives of $D_{\lambda,k}$, $T_{\lambda,k}$ with respect to $\lambda$.
Recalling the explicit expression~\eqref{eq:eigenv_expr} of $\mu_1(\lambda,k)$ and differentiating, we find that $\pa\mu_1/\pa\lambda< 0$ is equivalent to
$$
T'_{\lambda,k}-\frac{T_{\lambda,k}\,T'_{\lambda,k}-2\,D'_{\lambda,k}}{\sqrt{T_{\lambda,k}^2-4D_{\lambda,k}}}\,<\,0\,,
$$
which can be rewritten as
\begin{equation}
\label{eq:cond1}
\left(\sqrt{T_{\lambda,k}^2-4D_{\lambda,k}}-T_{\lambda,k}\right)T'_{\lambda,k}+2D'_{\lambda,k}\,<\,0\,.
\end{equation} 
Using the expression~\eqref{eq:eigenv_expr} of $\mu_1$ in terms of $T_{\lambda,k}$, $ D_{\lambda,k}$, we find out that at the points where $\mu_1(\lambda,k)$ vanishes it holds
\begin{equation}
\label{eq:cond_mu=0}
\sqrt{T_{\lambda,k}^2-4\,D_{\lambda,k}}\,=\,T_{\lambda,k}-4\,, \quad\hbox{ or equivalently } \quad \frac{D_{\lambda,k}}{2}\,=\,T_{\lambda,k}-2\,.
\end{equation}
Therefore, when $\lambda=\lambda_k$, condition~\eqref{eq:cond1} can be rewritten as 
\begin{equation}
\label{eq:cond2}
2\,T'_{\lambda_k,k}\,-\,D'_{\lambda_k,k}\,>\,0\,.
\end{equation}
In order to write condition~\eqref{eq:cond2} more explicitly, we compute $T'_{\lambda,k}$ and $D'_{\lambda,k}$: 
\begin{align*}
T'_{\lambda,k}\,&=\,R^2\frac{1-\lambda^2}{\lambda^3}\frac{k^2}{\sinh^2\omega}\,-\,2\,\frac{k}{\lambda}\,\coth\omega\,+\,2\,\frac{1-2R^2+\lambda^2}{\lambda(1-\lambda^2)}
\\
&=\,\frac{R^2}{\lambda^3(1-\lambda^2)}\left[(1-\lambda^2)^2 \frac{k^2}{\sinh^2\omega}\,+\,2(R^2-\lambda^2-1)(1-\lambda^2)k\coth\omega\,+\,2(1-R^2-R^2\lambda^2+\lambda^4)\right],
\\
D'_{\lambda,k}\,&=\,-2\,R^2\,\frac{k^2-1}{\lambda^3(1-\lambda^2)}\left[(R^2-\lambda^2)^2\,+\,(1-R^2)^2\right]\,.
\end{align*}

Now we specialize these formulas at the points where $\mu_1$ vanishes. To this end, from~\eqref{eq:cond_mu=0} and the explicit expressions~\eqref{eq:TD} for $T_{\lambda,k}$ and $D_{\lambda,k}$, we get
\begin{equation}
\label{eq:cond3}
2R^2_k(1-\lambda_k^2)k\coth\omega_k\,=\,(R^2_k-\lambda_k^2)(1-R^2_k)k^2\,+\,(R^2_k+\lambda_k^2)(1+R^2_k)\,,
\end{equation}
where we have used the notation $R_k=R(\lambda_k)$ and $\omega_k=\omega(\lambda_k)$ (recall that both $R$ and $\omega$ are functions of $\lambda$).
Starting from the above formulas for $T'_{\lambda,k}$ and $D'_{\lambda,k}$, recalling that $1/\sinh^2\omega=\coth^2\omega-1$ and plugging in identity~\eqref{eq:cond3}, with some computations we can rewrite condition~\eqref{eq:cond2} as
\begin{equation}
\label{eq:cond4}
4R^2_k(1-\lambda_k^2)k^3\coth\omega_k\,+\,2(\lambda_k^2-3 R_k^4-3R^2_k\lambda_k^2-3R^2_k)k^2\,
+\,3R_k^4+R^2_k\lambda_k^2+R^2_k-\lambda_k^2 \,>\, 0\,.
\end{equation}
Since $R_k>\lambda_k$, the $0$-th order term of~\eqref{eq:cond4} is positive.
 It follows that, in order for~\eqref{eq:cond4} to hold, it is sufficient to prove
\begin{equation}
\label{eq:cond5}
2R^2_k\left(1-\lambda_k^2\right)k\coth\omega_k\,+\,\lambda_k^2-3R_k^4-3R^2_k\lambda_k^2-3R^2_k\,>\, 0\,.
\end{equation}
In order to prove this inequality, we first need an estimate for $k\coth\omega_k$. To obtain it, we start from~\eqref{eq:cond_mu=0} and recall~\eqref{eq:delta_estimate} to get
\begin{align*}
T_{\lambda_k,k}\,&=\,4+\sqrt{T_{\lambda_k,k}^2-4D_{\lambda_k,k}}
\\
&>\,4+\left(\frac{c_{\lambda_k}^i}{\lambda_k}-c_{\lambda_k}^o\right)k\coth\omega_k\,+\,\left(\frac{c_{\lambda_k}^i}{\lambda_k}+c_{\lambda_k}^o\right)\,.
\end{align*}
We can then exploit the explicit expression~\eqref{eq:TD} for $T_{\lambda,k}$ to obtain, with some computations, the following estimate for $k\coth\omega_k$:
$$
k\coth\omega_k\,>\,\frac{R^2_k+\lambda_k^2}{\lambda_k^2(1-R^2_k)}\,.
$$
We are now ready to prove~\eqref{eq:cond5}. Using the inequality for $k\coth\omega_k$ that we just found, we can estimate the left hand side of~\eqref{eq:cond5} as follows:
\begin{align*}
2R^2_k(1-\lambda_k^2)k\coth\omega_k\,+\,\lambda_k^2-3R_k^4&-3R^2_k\lambda_k^2-3R^2_k\,>
\\
&>\,\frac{2R^2_k(1-\lambda_k^2)(R^2_k+\lambda_k^2)}{\lambda_k^2(1-R^2_k)}\,+\,\lambda_k^2-3R_k^4-3R^2_k\lambda_k^2-3R^2_k
\\
&=\,\frac{2R_k^4-R^2_k\lambda_k^2+\lambda_k^4-2R_k^4\lambda_k^2-6R^2_k\lambda_k^4+
3R_k^6\lambda_k^2+3R_k^4\lambda_k^4}{\lambda_k^2(1-R^2_k)}\,.
\end{align*}
Notice that
\begin{multline*}
2R_k^4-R^2_k\lambda_k^2+\lambda_k^4-2R_k^4\lambda_k^2-6R^2_k\lambda_k^4
+3R_k^6\lambda_k^2+3R_k^4\lambda_k^4
\,=\,
\\
\,=\,R^2_k(1+3\lambda_k^2)(R^2_k+R^2_k\lambda_k^2-2\lambda_k^2)+3R^2_k\lambda_k^2(1-R^2_k)^2+(R^2_k-\lambda_k^2)^2,
\end{multline*}
and since
\begin{equation*}
R^2_k+R^2_k\lambda_k^2-2\lambda_k^2\,=\,\lambda_k^2 \left(\frac{1}{\lambda_k}c_{\lambda_k}^i-c^o_{\lambda_k}\right)>\lambda_k^2 \left(c_{\lambda_k}^i-c^o_{\lambda_k}\right)>0\,,
\end{equation*}
it follows that inequality~\eqref{eq:cond5} is in force, as wished.
\end{proof}

We are finally ready to state the main proposition, that collects all the properties of the eigenvalues $\mu_1(\lambda,k)$, $\mu_2(\lambda,k)$ that we need.

\begin{proposition}
\label{pro:main_bifurcation}
The following properties hold.
\begin{enumerate}
\item[$(i)$] For any $\lambda\in(0,1)$, we have $\mu_1(\lambda,0)<\mu_2(\lambda,0)<0$, $\mu_1(\lambda,1)=-2$ and $\mu_2(\lambda,1)=0$.
\item[$(ii)$] $\mu_2(\lambda,k)>0$ for any $\lambda\in(0,1)$ and any integer $k\geq 2$.
\item[$(iii)$] For any integer $k\geq 2$, there exists a unique value $\lambda_k\in(0,1)$ such that $\mu_1(\lambda_k,k)=0$. Furthermore $\pa\mu_1/\pa\lambda(\lambda_k,k)<0$.
\item[$(iv)$] The sequence $\{\lambda_k\}_{k\geq 2}$ is monotonically increasing with $\lim_{k\to+\infty}\lambda_k=1$.
\item[$(v)$] $\lim_{k\to+\infty}\mu_i(\lambda,k)/k$ is finite and positive for all $\lambda\in(0,1)$, for $i=1,2$.
\end{enumerate}
\end{proposition}

\begin{proof}
We have already noticed that $\mu_1(\lambda,1)=-2$ and $\mu_2(\lambda,1)=0$ for all $\lambda\in(0,1)$. Since we also know from Lemma~\ref{le:monotonicity_k} that $\mu_1$ and $\mu_2$ are strictly increasing in $k$, points $(i)$ and $(ii)$ follow at once. 

Concerning the first eigenvalue $\mu_1$, we have already shown that for any $k\geq 2$ it holds 
$$
\lim_{\lambda\to 0}\mu_1(\lambda,k)\,=\,k-1\,>\,0\,,\quad \hbox{ and }\quad\lim_{\lambda\to 1}\mu_1(\lambda,k)\,=\,-2\,<\,0\,.
$$
Since $\mu_1$ is a continuous function of $\lambda$ for any fixed $k\geq 2$, it follows that there exists at least one value $\lambda_k\in(0,1)$ such that $\mu_1(\lambda_k,k)=0$. On the other hand, we know from Lemma~\ref{le:monotonicity_lambda} that the derivative of $\mu_1$ with respect to $\lambda$ has to be strictly negative at each point where $\mu_1$ vanishes. It follows that, once the function $\mu_1$ becomes negative, it cannot become positive again for larger values of $\lambda$. In other words, there can only be a single value $\lambda=\lambda_k$ at which $\mu_1$ vanishes. This proves point $(iii)$ of the proposition. 

The fact that $\{\lambda_k\}_{k\geq 2}$ is monotonically increasing follows immediately from Lemma~\ref{le:monotonicity_k}, thus, in order to prove point $(iv)$, it is enough to show that $\lambda_k$ goes to $1$ as $k\to\infty$. To do this, we argue by contradiction. Suppose that $\lim_{k\to\infty}\lambda_k=\lambda_\infty<1$ (notice that the limit exists because $\{\lambda_k\}$ is a monotone sequence). Recalling that formula~\eqref{eq:cond3} is in force at the point $\lambda=\lambda_k$, recalling also that $R$ and $\omega$ are both functions of $\lambda$ and setting $\omega_k=\omega(\lambda_k)$ and $R_k=R(\lambda_k)$, we have that it must hold
\begin{equation}
\label{eq:lambda_k_cond}
2R^2_k(1-\lambda_k^2)\coth\omega_k\,=\,(R^2_k-\lambda_k^2)(1-R^2_k)k\,+\,(R^2_k+\lambda_k^2)(1+R^2_k)\frac{1}{k}\,.
\end{equation}
Since $\lambda_k$ converges to a value $\lambda_\infty<1$,  it follows that $R_k$ also converges to a value $R_\infty$ such that $\lambda_\infty<R_\infty<1$, whereas $\coth\omega_k=(1+\lambda_k^{2k})/(1-\lambda_k^{2k})\to 1$ as $k\to +\infty$. As a consequence, the left hand side of~\eqref{eq:lambda_k_cond} converges to a finite value as $k\to\infty$, whereas the right hand side goes to infinity. This is a contradiction, as wished.

Finally, it remains to prove point $(v)$.
Recalling~\eqref{eq:eigenv_expr}, we have
\begin{align*}
\lim_{k\to +\infty}\frac{\mu_1(\lambda,k)}{k}\,&=\,\frac{1}{2}\lim_{k\to +\infty}\left\{\frac{T_{\lambda,k}}{k}\left[1-\sqrt{1-4\frac{D_{\lambda,k}}{T^2_{\lambda,k}}}\right]\right\}\,,
\\
\lim_{k\to +\infty}\frac{\mu_2(\lambda,k)}{k}\,&=\,\frac{1}{2}\lim_{k\to +\infty}\left\{\frac{T_{\lambda,k}}{k}\left[1+\sqrt{1-4\frac{D_{\lambda,k}}{T^2_{\lambda,k}}}\right]\right\}\,,
\end{align*} 
We now notice that for any fixed $\lambda$ it holds $\coth\omega=(1+\lambda^{2k})/(1-\lambda^{2k})\to 1$ as $k\to+\infty$. As a consequence, from the explicit expressions~\eqref{eq:TD} for $T_{\lambda,k}$ and $D_{\lambda,k}$, we easily compute
$$
\lim_{k\to+\infty}\frac{T_{\lambda,k}}{k}\,=\,\frac{R^2(1-\lambda^2)}{\lambda^2}\,, \quad\hbox{ and }\quad
\lim_{k\to+\infty}\frac{D_{\lambda,k}}{T^2_{\lambda,k}}\,=\,\frac{\lambda^2(R^2-\lambda^2)(1-R^2)}{R^4(1-\lambda^2)^2}\,.
$$
The desired result follows easily.
\end{proof}

The properties described in Proposition~\ref{pro:main_bifurcation} are the ones needed in order to be able to invoke the Crandall-Rabinowitz Bifurcation Theorem~\cite{Cra_Rab} to prove Theorem~\ref{thm:bifurcation}. Now the proof follows exactly the same strategy highlighed in~\cite[Section~5]{Kam_Sci}. Let us just recall briefly the main steps to show how the properties $(i)$-$(v)$ of the eigenvalues come into play. 

The first step is to restrict the functional $F_\lambda$ to functions invariant under a suitable group $G$ of isometries of $\R^2$. We can choose as $G$ any subgroup of the orthogonal group $O(2)$ such that the eigenvalues $\{\sigma_i\}_{i\in\N_0}$ of $-\Delta_{\Sph^1}$ have multiplicity $1$ and satisfy $\sigma_0=0$, $\sigma_1>1$. For instance, one possible choice is the group $G\cong\mathbb{Z}^2\times\mathbb{Z}^2$ acting on $\R^2$ by reflections along the two coordinate axes, in which case the eigenvalues form the sequence $\{(2i)^2\}_{i\in\N_0}$, corresponding to the eigenfunctions $\cos(2i\theta)$.

Let us fix such a $G$ once and for all, and let $\{\sigma_i\}_{i\in\N_0}$ be the corresponding sequence of eigenvalues. For every $i\in\N_0$ let $Y_i$ be the unique $G$-invariant unit $L^2(\Sph^1)$-norm  eigenfunction of $-\Delta_{\Sph^1}$ corresponding to the eigenvalue $\sigma_i$ and let $W_i={\rm Span}\{(Y_i,0),(0,Y_i)\}$. 

Fix now $k\in\N$. 
Recall from Proposition~\ref{pro:main_bifurcation}-($iii$) that there exists one value $\lambda_{\sigma_k}$ such that $\mu_1(\lambda_{\sigma_k},\sigma_k)=0$, let $F_k:=F_{\lambda_{\sigma_k}}$ be the operator defined as in~\eqref{eq:F} and let $L_k:=L_{\lambda_{\sigma_k}}$ be its linearization. Let us also denote by $\mathscr{C}_G^{\kappa,\alpha}(\Sph^1)$ the H\"older space of $\mathscr{C}^{\kappa,\alpha}(\Sph^1)$-functions that are $G$-invariant.
It is easily seen that the image of $G$-invariant functions via $F_k$ and $L_k$ is still $G$-invariant. We can then consider the restrictions
\begin{align*}
F_k&:\ \ \ \ \ \ \ U\ \ \ \ \ \ \longrightarrow \ \big(\mathscr{C}_G^{1,\alpha}(\Sph^1)\big)^2\,,
\\
L_k&:\big(\mathscr{C}_G^{2,\alpha}(\Sph^1)\big)^2\longrightarrow\  \big(\mathscr{C}_G^{1,\alpha}(\Sph^1)\big)^2\,,
\end{align*}
where $U\subseteq \big(\mathscr{C}_G^{2,\alpha}(\Sph^1)\big)^2$ is a small neighborhood of ${\bf 0}$.
In order to apply the Bifurcation Theorem to $F_k$, it is sufficient to show that ${\rm Ker}(L_k)$ has dimension $1$, that ${\rm Im}(L_k)$ is closed with codimension $1$ and that $\pa L_\lambda/\pa\lambda_{|_{\lambda=\lambda_{\sigma_k}}}({\bf z})\not\in{\rm Im}(L_k)$, where ${\bf z}$ is an element that spans the kernel of $L_k$. 

For $i\in\N_0$, let ${\bf z}_{i,1}$, ${\bf z}_{i,2}$ be eigenvectors of ${L_k}_{|_{W_i}}$ relative to the eigenvalues $\mu_1(\lambda_{\sigma_k},\sigma_i)$, $\mu_2(\lambda_{\sigma_k},\sigma_i)$ and orthonormal with respect to the scalar product $\langle\cdot\,,\,\cdot\rangle_{\lambda_{\sigma_k}}$ defined as in~\eqref{eq:scal_prod}. Denote $H_G^s(\Sph^1):=H^s(\Sph^1)\cap L_G^2(\Sph^1)$, where $L_G^2(\Sph^1)$ is the space of $G$-invariant $L^2(\Sph^1)$-integrable functions, and consider the map $(H_G^2(\Sph^1))^2\to(H_G^1(\Sph^1))^2$
\begin{equation}
\label{eq:map} 
\sum_{\ell=0}^\infty\big(a_{\ell,1}\, {\bf z}_{\ell,1}+a_{\ell,2}\, {\bf z}_{\ell,2}\big)\,\mapsto\,\sum_{\ell=0}^\infty\big(a_{\ell,1}\,\mu_1(\lambda_{\sigma_k},\sigma_\ell)\, {\bf z}_{\ell,1}+a_{\ell,2}\,\mu_2(\lambda_{\sigma_k},\sigma_\ell)\, {\bf z}_{\ell,2}\big)\,,
\end{equation}
This map coincides with $L_k$ on its domain, hence it defines an extension of $L_k$. Recall that the Sobolev norm on $H^s(\Sph^1)$ is equivalent to the norm
$$
\| f\|\,:=\,\sum_{j=0}^\infty(1+j^2)^s\,\|P_j(f)\|^2_{L^2}\,,
$$ 
where $P_j$ is the $L^2$-orthogonal projection on the subspace generated by spherical harmonics of degree $j$.
It follows then easily from the asymptotic behaviour of $\mu_1$ and $\mu_2$ proven in Proposition~\ref{pro:main_bifurcation}-$(v)$ that~\eqref{eq:map} is a continuous mapping. Furthermore, from Proposition~\ref{pro:main_bifurcation}-$(i),(ii),(iv)$, recalling that $\sigma_0=0$ and $\sigma_1>1$, it is clear that both  $\mu_1(\lambda_{\sigma_k},\sigma_i)$ and $\mu_2(\lambda_{\sigma_k},\sigma_i)$ are different from zero for every $i\in\N_0$ with the only exception of $\mu_1(\lambda_{\sigma_k},\sigma_k)$. As a consequence, we can write down the right inverse of $L_k$ as
$$
\sum_{\ell=0}^\infty\big(b_{\ell,1}\, {\bf z}_{\ell,1}+b_{\ell,2}\, {\bf z}_{\ell,2}\big)\,\mapsto\,\sum_{\ell=0,\ \ell\neq k}^\infty\Big(\frac{b_{\ell,1}}{\mu_1(\lambda_{\sigma_k},\sigma_\ell)}\,{\bf z}_{\ell,1}+\frac{b_{\ell,2}}{\mu_2(\lambda_{\sigma_k},\sigma_\ell)}\, {\bf z}_{\ell,2}\Big)\,+\,\frac{b_{\ell,2}}{\mu_2(\lambda_{\sigma_k},\sigma_k)}\, {\bf z}_{k,2}\,.
$$
Again from  Proposition~\ref{pro:main_bifurcation}-$(v)$ we deduce that this inverse is also continuous.
It follows that the map~\eqref{eq:map}, restricted to elements ${\bf v}$ satisfying $\langle {\bf v},{\bf z}_{k,1}\rangle_{\lambda_{\sigma_k}}=0$, is an isomorphism.
Since~\eqref{eq:map} is an extension of $L_k$, we can then expect, and indeed it can be proven with some work, that $L_k$ is an isomorphism as well when we restrict to elements orthogonal to ${\bf z}_{k,1}$. It follows immediately that ${\rm Ker}(L_k)$ has dimension $1$, generated by ${\bf z}_{k,1}$, whereas ${\rm Im}(L_k)$ is the space orthogonal to ${\bf z}_{k,1}$ with respect to the scalar product $\langle\cdot,\cdot\rangle_{\lambda_{\sigma_k}}$, and thus it is closed with codimension $1$. Finally, we have
$$
{\frac{\pa L_\lambda}{\pa\lambda}}_{|_{\lambda=\lambda_{\sigma_k}}}({\bf z}_{k,1})\,=\,\frac{\pa \mu_1}{\pa\lambda}(\lambda_{\sigma_k},\sigma_k)\,{\bf z}_{k,1}\,,
$$
which does not belong to ${\rm Im}(L_{\lambda_k})$ thanks to Proposition~\ref{pro:main_bifurcation}-$(iii)$. These are the properties that we needed in order to invoke the Crandall-Rabinowitz Bifurcation Theorem~\cite{Cra_Rab} (see also~\cite[Theorem~7.1]{Kam_Sci}), which tells us that, for sufficiently small $\ep>0$, $\delta>0$, there exists a smooth curve
\begin{align*}
(-\ep,\ep)&\longrightarrow\left(\mathscr{C}^{2,\alpha}_G(\Sph^1)\right)^2\times (\lambda_{\sigma_k}-\delta,\lambda_{\sigma_k}+\delta)
\\
s\ \ \ &\longmapsto\ \ \ \ \ \ \ \ \ ({\bf w}(s)\,,\,\lambda(s))
\end{align*}
with ${\bf w}(0)={\bf 0}$, $\lambda(0)=\lambda_{\sigma_k}$, $\langle{\bf w}(s)\,,\,{\bf z}_{k,1}\rangle_{\lambda_{\sigma_k}}= 0$ and $F_{\lambda(s)}({\bf v}(s))=0$, where ${\bf v}(s)=s({\bf w}(s)+{\bf z}_{k,1})$. It follows that the corresponding functions $u_{\lambda(s)}^{{\bf v}(s)}:\Omega_{\lambda(s)}^{{\bf v}(s)}\to\R$ form a $1$-parameter family of solutions to problem~\eqref{eq:prob_SD} having gradient constantly equal to $c_{\lambda(s)}^o$ on the outer boundary component and constantly equal to $c_{\lambda(s)}^i$ on the inner boundary component. This proves Theorem~\ref{thm:A_general}.



\renewcommand{\theequation}{A-\arabic{equation}}
\renewcommand{\thesection}{A}
\setcounter{equation}{0}  
\setcounter{theorem}{0}  

\section*{Appendix: expansions of $W$ and $W_R$ about \texorpdfstring{${\rm MAX}(u)$}{MAX(u)}}  
\label{sec:expansion}


In this appendix, we collect some basic, though very important expansions of the functions $W = |\na u|^2$ and $W_R= |\na u_R|^2 \circ \Psi$ about ${\rm MAX}(u)$. They have been invoked in the proofs of Theorem~\ref{thm:min_pr_SD} and of Proposition~\ref{pro:area_lower_bound_mon}.

\begin{lemma}
\label{le:Wm_umax}
Let $(\Om,u)$ be a solution to problem~\eqref{eq:prob_SD}, let $N$ be a connected component of $\Omega\setminus{\rm MAX}(u)$,  and let $R = R(N) \in [0,1)$ be the expected core radius associated with the region $N$. Also assume that  Normalization~\ref{norm} is in force. Then for every $p\in{\rm MAX}(u)$ it holds
$$
\lim_{x\to p,\,x\in N}\,\frac{W_R}{\umax-u}\,=\,4\,,
$$
where $W_R$ be the function defined in $N$ by~\eqref{eq:Wm}.
\end{lemma}

\begin{proof}
Let us remember that $W_R$, $\umax$ and $u$ can be rewritten explicitly in terms of $R$ and $\Psi$ via the following formul\ae:
$$
W_R\,=\,\left(\frac{\Psi^2-R^2}{\Psi}\right)^{\!\!2}\,,\quad 2\,\umax\,=\,1\,-\,R^2\,+\,R^2\,\log\,R^2\,,\quad 2\,u\,=\,1\,-\,\Psi^2\,+\,2\,R^2\,\log\,\Psi\,.
$$
Therefore, we have
$$
\lim_{x\to p,\,x\in N}\frac{W_R}{\umax-u}\,=\,
\lim_{\Psi\to R}\,\,\frac{2\,(\Psi^2\,-\,R^2)^2}{\Psi^2\left(\Psi^2-R^2+R^2\log R^2-2R^2\log\Psi\right)}\,.
$$
Setting $z=\Psi^2-R^2$, this limit can be easily computed with the following Taylor expansion:
\begin{align}
\notag
\frac{W_R}{\umax-u}\,&=\,
\frac{2\,z^2}{(R^2+z)\left[z-R^2\log(1+\frac{z}{R^2})\right]}
\\
\notag
&=\,\frac{2\,z}{(R^2+z)\left[1-\frac{R^2}{z}(\frac{z}{R^2}-\frac{1}{2}\frac{z^2}{R^4}+\frac{1}{3}\frac{z^3}{R^6}+\mathcal{O}(z^4))\right]}
\\
\notag
&=\,\frac{2\,z}{\frac{1}{2}z+\frac{1}{6}\frac{z^2}{R^2}+\mathcal{O}(z^3)}
\\
\label{eq:est_le1}
&=\,4-\frac{4}{3}\,\frac{z}{R^2}+\mathcal{O}(z^2)\,.
\end{align}
The desired statement follows at once.
\end{proof}

Notice that in the above proof we have actually shown a more precise estimate. Let us rephrase it in a more convenient way, as it will be helpful in the proof of the next lemma. Expanding $\umax-u$ in terms of $z$, we have
$$
\umax-u\,=\,\frac{1}{4}\,\frac{z^2}{R^2}\,+\mathcal{O}(z^3)\,.
$$
Inverting this relationship yields
$$
z\,=\,\pm\,2 R\sqrt{\umax-u}\,+\,\mathcal{O}(\umax-u)\,.
$$
The sign $\pm$ appears here because $z=\Psi^2-R^2$ is positive if $\overline\tau(N)<\sqrt{2}$ and it is negative if $\overline\tau(N)>\sqrt{2}$.
Therefore, the expansion~\eqref{eq:est_le1} above can be rewritten as
\begin{align}
\notag
W_R\,&=\,4(\umax-u)-\frac{4}{3}(\umax-u)\frac{z}{R^2}+\mathcal{O}\left((\umax-u)\,z^2\right)
\\
\label{eq:Wm_umax}
&=\,4(\umax-u)\mp\frac{8}{3 R}(\umax-u)^{3/2}+\mathcal{O}\left((\umax-u)^2\right) \,,
\end{align}
where the $-$ sign holds on regions where  $\overline\tau(N)<\sqrt{2}$, and the $+$ sign holds on regions where  $\overline\tau(N) > \sqrt{2}$.

In the following lemma, we provide more refined expansions for both $W$ and $W_R$ in a neighborhood of the top stratum of ${\rm MAX}(u)$.

\begin{lemma}
\label{le:WWm_expansions}
Let $(\Om,u)$ be a solution to problem~\eqref{eq:prob_SD}, let $N$ be a connected component of $\Omega\setminus{\rm MAX}(u)$,  and let $R = R(N) \in [0,1)$ be the expected core radius associated with the region $N$. Also assume that  Normalization~\ref{norm} is in force and denote by $\Sigma_N$ the $1$-dimensional top stratum of ${\rm MAX}(u) \cap \overline{N}$. Then, at any point $p \in \Sigma_N$, it holds
\begin{align*}
W\,&=\,4\,r^2\big[1\,+\, \kappa(p)\,r\big]\,+\,\mathcal{O}(r^4)\,,
\\
W_R\,&=\,4\,r^2\,\left[1\,+\,\left(\frac{\kappa(p)}{3}\,-\,\frac{2}{3 R}\right)r\right]+\mathcal{O}(r^4)\,, \quad \hbox{if $\overline\tau(N) <\sqrt{2}$}\,,
\\
W_R\,&=\,4\,r^2\,\left[1\,+\,\left(\frac{\kappa(p)}{3}\,+\,\frac{2}{3 R}\right)r\right]+\mathcal{O}(r^4)\,, \quad \hbox{if $\overline\tau(N) >\sqrt{2}$} \, ,
\end{align*}
where $\kappa (p)$ is the curvature of $\Sigma_N$ at $p$ computed with respect to the exterior unit normal, and $r(x)={\rm dist}(x,\Sigma_N)$ denotes the distance from $\Sigma_N$.
\end{lemma}

\begin{proof}
From~\cite[Theorem 3.1]{Bor_Chr_Maz} we have the following expansion:
$$
u\,=\,\umax\,-\,r^2\,-\,\frac{\kappa(p)}{3}\,r^3\,+\,\mathcal{O}(r^4)\,.
$$
In particular, $\D u$ satisfies
$$
\D u\,=\,-r\,\big[2\,+\,\kappa(p)\, r\big]\frac{\pa}{\pa r}+\mathcal{O}(r^3)
$$
and we get
$$
W\,=\,|\D u|^2\,=\,4\,r^2\,+\,4\,\kappa(p)\,r^3\,+\,\mathcal{O}(r^4)\,=\,4\,r^2\big[1\,+\, \kappa(p)\,r\big]\,+\,\mathcal{O}(r^4)\,.
$$
Concerning $W_R$, let us observe that formula~\eqref{eq:Wm_umax} provides us with an expansion of $W_R$ in terms of $\umax-u$. It follows immediately that
\begin{align*}
W_R\,&=\,4(\umax-u)\mp\frac{8}{3 R}(\umax-u)^{3/2}+\mathcal{O}\left((\umax-u)^2\right)
\\
&=\,4\,r^2\,\mp\,\frac{8}{3 R}\,r^3\,+\,\frac{4}{3}\,\kappa(p)\,r^3\,+\,\mathcal{O}(r^4)
\\
&=\,4\,r^2\,\left[1\,+\,\left(\frac{\kappa(p)}{3}\, \mp\,\frac{2}{3 R}\right)r\right]+\mathcal{O}(r^4)\,,
\end{align*}
where the sign ambiguity is the one specified in the statement. This concludes the proof.
\end{proof}

In particular, the above lemma implies that the quotient $W/W_R$ tends to one as we approach $\Sigma_N$. We will also need an integral estimate of this quantity, and that is the content of the next lemma.

\begin{lemma}
\label{le:W_Wm}
Let $(\Om,u)$ be a solution to problem~\eqref{eq:prob_SD}, let $N$ be a connected component of $\Omega\setminus{\rm MAX}(u)$,  and let $R = R(N) \in [0,1)$ be the expected core radius associated with the region $N$.
Also assume that Normalization~\ref{norm} is in force. If $\Sigma_N = \pa \overline{N} \cap {\rm MAX}(u) \neq \emptyset$ is a smooth simple closed curve, then it holds
$$
\lim_{t\to \umax^-}\int_{\{u=t\}\cap N}\sqrt{\frac{W}{W_R}}\,\rmd\sigma\,\geq\,|\Sigma_N|\,,
$$
where $W=|\D u|^2$ and $W_R$ is the function defined in $N$ by~\eqref{eq:Wm}.
\end{lemma}

\begin{proof}
Since we know from Lemma~\ref{le:WWm_expansions} that $W/W_R$ tends to $1$, it is enough to prove that 
$$
\limsup_{t\to\umax^-}|\{u=t\}\cap N|\,\geq\,|\Sigma_N|\,.
$$
The proof of this inequality follows exactly the same strategy employed in~\cite[Proposition~5.4]{Bor_Maz_2-I}, so we leave the details to the interested reader. The idea is that, since $\Sigma_N$ is a compact smooth hypersurface, the (signed) distance function $r$ is also smooth in a sufficiently small neighborhood and we can project $\{u=t\}\cap N$ onto $\Sigma_N$ using the flow of $\D r$. We can then prove the bound on $|\Sigma_N|$ directly from the definition of Hausdorff measure, showing that the length of the curves does not collapse along the projection.
\end{proof}

With some more work, one can actually prove that the equality holds in Lemma~\ref{le:W_Wm} by showing that $\D u/|\D u|$ is smooth in a neighborhood of $\Sigma_N$ and using its gradient flow to show that $|\{u=t\}\cap N|$ converges to $|\Sigma_N|$. However, the inequality is sufficient for the purposes of the present manuscript.

\subsection*{Acknowledgements}
{\em The authors would like to thank X. Cabr\'e, R. Magnanini,  A. Roncoroni, L. Sciaraffia for their interest in our work and for stimulating discussions during the preparation of the manuscript. The authors are members of the Gruppo Nazionale per l'Analisi Matematica, la Probabilit\`a e le loro Applicazioni (GNAMPA) of the Istituto Nazionale di Alta Matematica (INdAM).
}

\bibliographystyle{abbrv}
\bibliography{biblio}

\end{document}